\documentclass[11pt,reqno]{amsart}
\usepackage{amsmath,amsfonts,amsthm,amssymb}
\usepackage[usenames]{color}
\newtheorem{Theorem}{Theorem}[section]
\newtheorem{Lemma}{Lemma}[section]
\newtheorem{Proposition}{Proposition}[section]

\theoremstyle{definition}
\newtheorem{Definition}{Definition}[section]
\theoremstyle{remark}

\numberwithin{equation}{section}

\renewcommand{\u}{{\bf u}}
\newcommand{\g}{{\bf g}}

\renewcommand{\H}{{\bf H}}
\renewcommand{\div}{{\rm div}}
\newcommand{\R}{{\mathbb R}}
	\newcommand{\q}{{\bf q}}

\def\hf1{^\f{1}{1-\xi^2}}

\def\be{\begin{equation}}
\def\en{\end{equation}}
\def\bs{\begin{split}}
\def\es{\end{split}}
\newcommand{\PP}{\mathbb{P}}

\renewcommand{\d}{{\bf d}}

\newcommand{\ww}{{\bf w}}
\newcommand{\f}{{\bf f}}

\newcommand{\loc}{\rm loc}

\newcommand{\essinf}{{\rm ess\inf }}

\author{Hengrong Du, \ Yimei Li, \ Changyou Wang}
\address{Department of Mathematics, Purdue University, West Lafayette, IN, 47907, USA.}
\email{du155@purdue.edu}
\address{Department of Mathematics, Beijing Normal University, Beijing, PRC.} \email{lym@mail.bnu.edu.cn}
\address{Department of Mathematics, Purdue University, West Lafayette, IN, 47907, USA.}
\email{wang2482@purdue.edu}

\title[Non-isothermal LCF]
{Weak solutions of  non-isothermal nematic liquid crystal flow in dimension three}

\keywords{Non-isothermal nematic liquid crystals, Ginzburg-Landau approximation, Entropy inequalities}
\subjclass[2000]{35A05, 76A10, 76D03.}
\date{\today}
\everymath{\displaystyle}
\begin{document}

\begin{abstract}
	For any smooth domain $\Omega\subset \R^3$, we establish the existence of a global  weak solution $(\u,\d, \theta)$ to the simplified, non-isothermal Ericksen-Leslie system modeling the hydrodynamic motion of nematic liquid crystals with variable temperature for any initial and boundary data $(\u_0, \d_0, \theta_0)\in\mathbf{H}\times H^1(\Omega, \mathbb{S}^2)\times L^1(\Omega)$, with $ \d_0(\Omega)\subset\mathbb{S}_+^2$ (the upper half sphere) and $\displaystyle\essinf_\Omega \theta_0>0$.
	
\end{abstract}

\maketitle
%%%%%%%%%%%%%%%%%%%%%%%%%%%%%%%%%%%%%%%%%%%%%%%%%%%%%%%%%%%%%%%%%%%%%%%%
\section{Introduction}
\indent The liquid crystal constitutes a state of matter which is intermediate between the solid and the liquid. In the nematic phase, molecules move like those in fluid, while they tend to reveal preferable orientations. A non-isothermal liquid crystal flow in the nematic phase can be described in terms of three physical variables: the velocity field $\u$ of
the underlying fluid, the director field $\d$ representing the averaged orientation of liquid crystal molecules, and the background temperature $\theta$. The evolution of the velocity field is governed by the incompressible Navier-Stokes system with stress tensors representing viscous and elastic effects. In the nematic case, the director field is driven by transported negative gradient flow of the Oseen-Frank energy functional which represents the internal microscopic damping \cite{DE, GP}.  We consider the non-isothermal setting in which the temperature is neither spatial nor temporal homogeneous and thus contributes to total dissipation of the whole system.

\indent A great deal of mathematical theories has been devoted to the study of nematic liquid crystals in the continuum formulation. In pioneering papers \cite{E1,E2,Leslie} Ericksen and Leslie have put forward a PDE model based on the principle of conservation laws and momentum balance. There has been extensive mathematical study of analytic issues of the simplified Ericksen-Leslie system. In 1989 Lin \cite{L1} first proposed a simplified Ericksen-Leslie model with one constant approximation for the Oseen-Frank energy:
$(\u, \d):\Omega\times\R_+\rightarrow \R^n\times\mathbb{S}^2$ solves
\begin{equation}
\left\{
\begin{array}{l}
\partial_t \u+\u\cdot \nabla\u+\nabla P=\mu\Delta \u-\nabla\cdot (\nabla \d\odot\nabla \d), \\
\nabla\cdot \u=0, \\
\partial_t\d+\u\cdot \nabla\d=\Delta\d+|\nabla \d|^2\d,
\end{array}
\right.
\label{eqn:1.2}
\end{equation}
where $\Omega\subset \R^n$ ($n=2$ or $3$), $P:\Omega\times \R_+\rightarrow\R$ denotes the pressure,  $\mu>0$ 
represents the viscosity constant of the fluid, and $(\nabla\d\odot\nabla\d)_{ij}=\sum_{k=1}^3 \partial_{x_i}\d^{(k)}\partial_{x_j}\d^{(k)}$ denotes the Ericksen stress tensor. {It is a system of the forced Navier-Stokes equation coupled with the transported harmonic map heat flow
to $\mathbb{S}^2$.  The readers can consult \cite{Temam} on the study of the Navier-Stokes equations and \cite{LW3}
for some recent developments on harmonic map heat flow.}
The rigorous mathematical analysis was initiated by Lin-Liu {\cite{LL1,LL2}} in which they established the well-posedness of so-called Ginzburg-Landau approximation of \eqref{eqn:1.2}: $(\u, \d):\Omega\times\R_+\rightarrow \R^n\times\R^3$ satisfies
\begin{equation}
  \left\{
    \begin{array}{l}
      \partial_t \u+\u\cdot \nabla\u+\nabla P=\mu\Delta \u-\nabla\cdot (\nabla \d\odot\nabla \d), \\
\nabla\cdot \u=0, \\
\partial_t\d+\u\cdot \nabla\d=\Delta\d+\frac{1}{\varepsilon^2}\left( 1-|\d|^2 \right)\d,
    \end{array}
    \right.
    \label{eqn:1.1}
  \end{equation}
where {$\varepsilon>0$} is the parameter of approximation. 
{They have obtained the existence of a unique, global strong solution in dimension $2$ and in dimension $3$ under large viscosity $\mu$. They have also studied  the
existence of suitable weak solutions and their partial regularity in dimension $3$,
which is analogous to the celebrated regularity theorem by Caffarelli-Kohn-Nirenberg \cite{CKN} (see also \cite{LIN}) for the dimension $3$ incompressible Navier-Stokes equation.}
  Later on Lin-Lin-Wang \cite{LLW} adopted a different approach to construct global {Leray-Hopf type weak solutions (see \cite{Leray})}
   for dimension $2$  to \eqref{eqn:1.2}  via the method of small energy regularity estimate. {Huang-Lin-Wang \cite{HLW} extended the works of \cite{LLW} to the general Ericksen-Leslie system by a blow up argument.}

The existence of global weak solution to \eqref{eqn:1.2} in dimension three is highly non-trivial due to the appearance of the super-critical nonlinear elastic stress term $\nabla\cdot(\nabla\d\odot\nabla\d)$. Some preliminary progress was made by Lin-Wang \cite{LW2}, where under the assumption that an initial configuration $\d_0$ lies in
the upper half sphere, i.e.,
 \begin{equation}
\d_0(\Omega)\subset \mathbb{S}^2_+:=\big\{y=(y^{1}, y^{2}, y^{3})\in \R^3: |y|=1,
\ {y}^{3}\geq 0\big\}.
\label{eqn:d0} \end{equation}
the existence of global weak solution was constructed by the Ginzburg-Laudau approximation method and
a delicate  blow-up analysis. { See \cite{LW1} for a review of recent progresses on the mathematical analysis of Ericksen-Leslie system.} 

Recently there has been considerable interest in the mathematical study for the hydrodynamics of 
non-isothermal nematic liquid crystals. Recall that 
a simplified, non-isothermal version of \eqref{eqn:1.1} can be described as follows.
Let $(\u,\d,\theta):\Omega\times\R_+\rightarrow\R^n\times\R^3\times\R_+$ solve
\begin{equation}
  \left\{
    \begin{array}{l}
      \partial_t \u+\u\cdot \nabla\u+\nabla P=\nabla\cdot(\mu(\theta)\nabla \u)-\nabla\cdot (\nabla\d\odot\nabla \d), \\
      \nabla\cdot \u=0, \\
      \partial_t\d+\u\cdot \nabla\d=\Delta \d+\frac{1}{\varepsilon^2}\left( 1-|\d|^2 \right)\d, \\
      \partial_t \theta+ \u\cdot \nabla\theta=-\nabla\cdot \q+\mu(\theta)|\nabla\u|^2+\big|\Delta \d+\frac{1}{\varepsilon^2}(1-|\d|^2)\d\big|^2,
    \end{array}
    \right.
    \label{eqn:1.3}
  \end{equation}
  where  $\q:\Omega\times\R_+\rightarrow \R^n$ is the heat flux.  Feireisl- Fr\'{e}mond-Rocca-Schimperna \cite{FFRS} proved the existence of a global weak solution to \eqref{eqn:1.3} in dimension $3$. Correspondingly, non-isothermal version of \eqref{eqn:1.2} reads $(\u, \d, \theta):\Omega\times\R_+\rightarrow\R^n\times\mathbb{S}^2\times\R_+$ solves
  \begin{equation}
  \left\{
  \begin{array}{l}
  \partial_t \u+\u\cdot \nabla\u+\nabla P=\nabla\cdot(\mu(\theta)\nabla \u)-\nabla\cdot (\nabla\d\odot\nabla \d), \\
  \nabla\cdot \u=0, \\
  \partial_t\d+\u\cdot \nabla\d=\Delta \d+|\nabla\d|^2\d, \\
  \partial_t \theta+ \u\cdot \nabla\theta=-\nabla\cdot \q+\mu(\theta)|\nabla\u|^2+\left|\Delta \d+|\nabla\d|^2\d\right|^2.
  \end{array}
  \right.
  \label{eqn:1.4}
  \end{equation}
Hieber-Pr\"{u}ss \cite{Hieber} have established the existence of a unique local $L^p-L^q$ strong solution to \eqref{eqn:1.4}, which can be extended to a global strong solution provided the initial data is close to an equilibrium state. For the general non-isothermal Ericksen-Leslie system, De Anna-Liu \cite{Anna} have obtained the existence of global strong solution in Besov spaces provided the Besov norm of the initial data is sufficiently small.  On $\mathbb{T}^2$, Li-Xin \cite{LX} have showed that there exists a global weak solution to \eqref{eqn:1.4}.  A natural question is that in  dimension $3$ whether \eqref{eqn:1.4} admits a global weak solution.
The main goal of this paper is to give a positive answer  under the additional assumption \eqref{eqn:d0}.  \\
\indent This paper is organized as follows. We devote Section $2$ to the derivation of thermodynamic consistency of a simplified, non-isothermal Ericksen-Leslie system for nematic liquid crystals. The weak formulation for \eqref{eqn:1.4} model is demonstrated in Section $3$. In Section $4$ we will establish the weak maximum principle for the free drifted Ginzburg-Landau heat flow with homogeneous Neumann boundary condition.  In Section $5$, we will establish a priori estimates and the existence of weak solutions to the non-isothermal Ericksen-Leslie system. In Appendix A, we provide a Faedo-Galerkin scheme for weak solutions to the non-isothermal model \eqref{eqn:1.3}.
\section{Thermal consistency of the non-isothermal nematic models}
\subsection{Non-isothermal Ginzburg-Landau approximation}
First we recall the equations of $\u$ and $\d$ in the non-isothermal Ginzburg-Laudau approximation \eqref{eqn:1.3}:
\begin{equation}
  \left\{
    \begin{array}{l}
      \partial_t\u+\u\cdot \nabla \u+\nabla P=\div\left(
      \mu(\theta)\nabla \u-\nabla \d\odot \nabla \d \right), \\
      \nabla\cdot \u=0,\\
      \partial_t\d+\u\cdot \nabla \d= \Delta \d-\f_\varepsilon(\d) ,
    \end{array}
    \right.
    \label{eqn:GinzburgLandau}
  \end{equation}
  where $\f_\varepsilon(\d)=\partial_{\d} F_\varepsilon(\d)$,  $F_\varepsilon (\d)=\frac{(|\d|^2-1)^2}{4\varepsilon^2}$. 
  
 The difference between \eqref{eqn:GinzburgLandau} and the isothermal case \eqref{eqn:1.1} is that the viscosity coefficient $\mu$ is a function of temperature $\theta$. Here the temperature plays a role as parameters both in the material coefficients and the heat conductivity coefficients, which is to be discussed later. To make the system \eqref{eqn:GinzburgLandau} a close system, we need the evolution equation for $\theta$. 
The equation of thermal dissipation is derived according to \textit{First and Second laws  of thermodynamics} \cite{SV}.  

  First we introduce  some basic concepts in thermodynamics.
  The internal energy density reads
  \begin{equation*}
    e_\varepsilon^{int}=\frac12{|\nabla \d|^2}+ F_\varepsilon(\d)+\theta,
    \label{}
  \end{equation*}
  and the Helmholtz free energy is given by
  \begin{equation*}
  \psi_\varepsilon=\frac12{|\nabla\d|^2}+ F_\varepsilon(\d)-\theta\ln \theta.
  \label{}
  \end{equation*}

Denote  the entropy  by {\color{red}$\eta$} in the \textit{Second law of thermodynamics},
which is determined by temperature through the Maxwell relation
  \begin{equation}
    \eta=-\frac{\partial \psi_\varepsilon}{\partial \theta}=1+\ln \theta.
    \label{2.4}
  \end{equation}
The internal energy can be obtained by (negative) Legendre transformation of free energy with respect to $\eta$, i.e.,
\begin{equation*}
  e^{int}_\varepsilon=\psi_\varepsilon+\eta\theta.
  \label{}
\end{equation*}
The heat flux $\q$ in the equations of both $\theta$ of \eqref{eqn:1.3} and \eqref{eqn:1.4}
satisfies the generalized Fourier law:
\begin{equation}
\q(\theta)=-k(\theta)\nabla\theta-h(\theta)(\nabla \theta\cdot\d)\d
\label{heat_flux}
\end{equation}
where $k(\theta)$ and $h(\theta)$ represent thermal conductivities.
The evolution of entropy can be written as follows.
\begin{equation}
  \eta_t+\u\cdot \nabla \eta=-\nabla\cdot \g+\Delta_\varepsilon,
  \label{2.7}
\end{equation}
where $\g$ is the entropy flux which is determined by the heat flux through the Clausius-Duhem relation
\begin{equation}
\q=\theta\g,
\label{2.8}
\end{equation}
 and the entropy production $\Delta_\varepsilon\ge 0$ is given by \eqref{Delta} below. \\
%Using the fact that $\eta^\varepsilon=1+\log \theta_\varepsilon$, formally we have the evolution of
%temperature
%\begin{equation}
%  \theta_\varepsilon_t+u^\varepsilon\cdot \nabla\theta_\varepsilon=\theta_\varepsilon\nabla\cdot \left( \frac{q^\varepsilon}{\theta_\varepsilon} \right)+\theta_\varepsilon\Delta_\varepsilon^*.
 % \label{}
%\end{equation}
%In order to obtain a closed system, we need to derive the expression for $\Delta^*$.

The thermal consistency of \eqref{eqn:1.3} is given by the following proposition.
\begin{Proposition}
	Suppose $(\u, \d, \theta)$ is a strong solution to \eqref{eqn:1.3}. Then\\
(1) {\rm{(}}\textit{First law of thermodynamics}{\rm{)}}. The total energy 
$e_\varepsilon^{total}=\frac12{|\u|^2}+e_\varepsilon^{int}$ is conservative. 
More precisely,  we have
	    \begin{equation}
	      \frac{D}{Dt}e_\varepsilon^{total}+\nabla\cdot (\Sigma+\q)=0,
	      \label{Sigma}
	    \end{equation}
	    where 
\begin{equation}\label{sigma}
\Sigma=P \u -\mu(\theta)\u\cdot \nabla \u+\nabla\d\odot\nabla\d\cdot \u-(\nabla \d)^T\frac{D \d}{Dt},
\end{equation}
and $\frac{D}{Dt}:=\frac{\partial}{\partial t}+\u\cdot \nabla$ denotes the material derivative.\\
(2) {\rm{(}}\textit{Second law of thermodynamics}{\rm{)}}. The entropy cannot decrease during any irreversible process, which means the entropy production $\Delta_\varepsilon$ is alway non-negative, i.e.,
	  \begin{equation}
	  	\Delta_\varepsilon=\frac{1}{\theta}\Big(\mu(\theta)|\nabla \u|^2+\big|\Delta\d+\frac{1}{\varepsilon^2}(1-|\d|^2)\d\big|^2-{\bf q}\cdot\nabla\theta\Big)\geq 0.
      \label{Delta}
	  \end{equation}
\end{Proposition}
\begin{proof}
 We first prove \eqref{Sigma}. By direct calculations, we have
\begin{equation}
\begin{array}{l}
\frac{D}{Dt}e_\varepsilon^{total}=\u\cdot \frac{D\u}{Dt}+ \nabla \d:\frac{D}{Dt}\nabla \d+ f_\varepsilon(\d)\cdot \frac{D\d}{Dt}+\frac{D\theta}{Dt}\\
=\u\cdot \div\left(-P I+ \mu(\theta)\nabla \u-\nabla \d\odot\nabla \d\right)
+\nabla \d:\nabla\frac{D \d}{Dt}-\nabla \d\odot \nabla \d:\nabla \u\\
\quad + \f_\varepsilon(\d)\cdot \frac{D\d}{Dt}-\nabla\cdot\q+\mu(\theta)|\nabla \u|^2+\big|\Delta\d+\frac{1}{\varepsilon^2}(1-|\d|^2)\d\big|^2\\
=\div\left( -P \u+\mu(\theta)\u\cdot\nabla \u-\nabla \d\odot \nabla \d\cdot \u \right)-\mu(\theta)|\nabla \u|^2+\nabla \d\odot \nabla \d:\nabla \u\\
\quad  +\div\big( (\nabla \d)^T \frac{D\d}{Dt} \big)-(\Delta \d-\f_\varepsilon(\d))\cdot \frac{D\d}{Dt}-\nabla \d\odot \nabla \d:\nabla \u-\nabla\cdot \q\\
\quad+\mu(\theta)|\nabla \u|^2+\big|\Delta\d+\frac{1}{\varepsilon^2}(1-|\d|^2)\d\big|^2\\
=\div\big( -P \u +\mu(\theta)\u\cdot \nabla \u-\nabla \d\odot \nabla \d\cdot \u+(\nabla \d)^T \frac{D\d}{Dt} \big)-\nabla\cdot\q\\
=-\div(\Sigma+\q).
\end{array}
\label{eqn:ApproxTotalEnergy}
\end{equation}
Note that \eqref{Delta} follows directly
from \eqref{2.4}, \eqref{2.7},  \eqref{eqn:1.3}$_4$, and \eqref{heat_flux}, i.e.
\begin{eqnarray*}
  \Delta_\varepsilon&=&\frac{1}{\theta}\left( \mu(\theta)|\nabla\u|^2+\left|\Delta \d-f_\varepsilon(\d)\right|^2 -{\bf q}\cdot\nabla\theta\right)\\
 &=& \frac{1}{\theta}\left( 
 \mu(\theta)|\nabla\u|^2+\left|\Delta \d-f_\varepsilon(\d)\right|^2 +k(\theta)|\nabla\theta|^2
 +h(\theta)|\nabla\theta\cdot{\bf d}|^2\right)\geq 0.
\label{}
\end{eqnarray*}
This completes the proof.
\end{proof}
\subsection{Non-isothermal simplified Ericksen-Leslie system}
As $\varepsilon$ tends to $0$,  due to the penalization effect of $F_\varepsilon(\d)$, formally 
the equation of $\d$  in (\ref{eqn:GinzburgLandau}) converges to
$$\partial_t \d+\u\cdot\nabla\d=\Delta \d+|\nabla\d|^2\d,$$
where $|\d|=1$. This is a ``transported gradient flow" of the Dirichlet energy 
$\frac{1}{2}\int_{\Omega}|\nabla \d|^2\,dx$ for maps $\d:\Omega\to \mathbb{S}^2$.

As in the previous section, we introduce the total energy for \eqref{eqn:1.4}: 
$$e^{total}=\frac12({|\u|^2}+{|\nabla \d|^2})+\theta,$$
and the entropy evolution equation:
  \begin{equation}
  \eta_t+\u\cdot \nabla \eta=-\nabla\cdot \g+\Delta_0,
  \label{2.13}
  \end{equation}
where $\Delta_0$ is the entropy production given by \eqref{therm2} below.

The thermal consistency of \eqref{eqn:1.4} is described by the following proposition.

\begin{Proposition}
	Suppose $(\u, \d, \theta)$ is a strong solution to \eqref{eqn:1.4}. Then
	
%	\begin{enumerate}
\noindent (1) {\rm{(}}\textit{First law of thermodynamics}{\rm{)}}.
The total energy is conservative, i.e.,
	    \begin{equation}
	      \frac{D}{Dt}e^{total}+\nabla\cdot (\Sigma+\q)=0,
	      \label{therm1}
	    \end{equation}
	    where $\Sigma=P \u -\mu(\theta)\u\cdot \nabla\u+\nabla\d\odot\nabla\d\cdot \u-(\nabla\d)^T\frac{D\d}{Dt}.$
	    
\noindent (2) {\rm{(}}\textit{Second law of thermodynamics}{\rm{)}}.
The entropy production $\Delta_0$ is non-{negative}, i.e.,
	     \begin{equation}
	     	\Delta_0=\frac{1}{\theta}
		\left(\mu(\theta) |\nabla u|^2+|\Delta{\bf d}+|\nabla{\bf d}|^2{\bf d}|^2
		-{\bf q}\cdot\nabla\theta\right)
		\geq 0.
	     	\label{therm2}
	     \end{equation}
\end{Proposition}
\begin{proof}
	From \eqref{eqn:1.4}, we can compute
 \begin{equation*}
 \begin{array}{l}
   \frac{De^{total}}{Dt}=\frac{D}{Dt}\big( \frac12({|\u|^2}+{|\nabla \d|^2})+\theta \big)\\
   =\u\cdot \frac{D\u}{Dt}+\nabla\d:\frac{D}{Dt}\nabla \d+\frac{D\theta}{Dt}\\
   =\u\cdot \div\left( -PI+\mu(\theta)\nabla \u-\nabla \d\odot \nabla \d \right)\\
   \quad +\nabla \d:\nabla \frac{D\d}{Dt}-\nabla \d\odot \nabla \d:\nabla \u-\nabla\cdot \q+\mu(\theta)|\nabla \u|^2+\left|\Delta \d+|\nabla \d|^2 \d\right|^2\\
   =\div\left( -P\u+\mu(\theta)\u\cdot \nabla \u-\nabla \odot \nabla \d\cdot \u \right)-\mu(\theta)|\nabla \u|^2+\nabla \d\odot \nabla \d:\nabla \u\\
   \quad +\div\big( (\nabla \d)^T\frac{D\d}{Dt} \big)-(\Delta \d+|\nabla \d|^2 \d)\cdot \Delta \d-\nabla \d\odot \nabla \d:\nabla \u\\
   \quad -\div \q+\mu(\theta)|\nabla \u|^2+|\Delta \d+|\nabla \d|^2 \d|^2
   \\
   =-\div (\Sigma+\q),
   \end{array}
 \end{equation*}
where we have used the fact $|\d|=1$ so that
$$(\Delta \d+|\nabla \d|^2 \d)\cdot \Delta \d=|\Delta \d+|\nabla \d|^2 \d|^2.$$
This implies \eqref{therm1}. From
the entropy equation \eqref{2.13}, Clausius-Duhem's relation \eqref{2.8},
the temperature equation in \eqref{eqn:1.4}, and \eqref{heat_flux},
we can show
\begin{eqnarray*}
\Delta_0&=&\frac{1}{\theta}\big( \mu(\theta)|\nabla\u|^2+\left|\Delta \d+|\nabla\d|^2\d\right|^2-{\bf q}\cdot\nabla\theta \big)\\
&=&\frac{1}{\theta}\big( \mu(\theta)|\nabla\u|^2+\left|\Delta \d+|\nabla\d|^2\d\right|^2
+k(\theta)|\nabla\theta|^2+h(\theta)|\nabla\theta\cdot{\bf d}|^2 \big)\ge 0.
\label{}
\end{eqnarray*}
This yields \eqref{therm2}.
\end{proof}

 \section{Weak formulation for Ericksen-Leslie system \eqref{eqn:1.4}}
Throughout this paper, we will assume
that $\mu$ is a continuous function,
 and $h, k$ are Lipschitz continuous functions, and 
 \begin{equation}
 0<\underline{\mu}\leq \mu(\theta)\leq \overline{\mu}, \quad 0<\underline{k}\leq k(\theta), h(\theta)\leq \overline{k} \quad \text{ for all }\theta>0,
 \label{eqn:coeffbound}
 \end{equation}
where $\underline{\mu}$, $\overline{\mu}$, $\underline{k}$, and $\overline{k}$ are positive
constants.
  We will impose the homogeneous boundary condition for $\u$:
 \begin{equation}
 \u|_{\partial \Omega}=0,\quad \frac{\partial\d}{\partial\nu}\big|_{\partial \Omega}=0,
 \label{eqn:Boundaryud}
 \end{equation}
 where $\nu$ is the outward unit normal vector field of $\partial\Omega$. 
 It is readily seen that
 \eqref{eqn:Boundaryud} implies that for $\Sigma$ given by \eqref{sigma}, it holds
 \begin{equation}
   \Sigma\cdot\nu|_{\partial \Omega}=0.
   \label{eqn:Boundaryud1}
 \end{equation}
 We will also impose the non-flux boundary condition for the temperature function so that
 the heat flux $\q$ satisfies
 \begin{equation}
 \q\cdot \nu|_{\partial \Omega}=0.
 \label{eqn:Boundaryq}
 \end{equation}
 Set
 $$\mathbf{H}=\text{ Closure of $C_0^\infty(\Omega;\R^3)\cap \left\{ v:\nabla\cdot v=0 \right\}$ in $L^2(\Omega;\R^3)$},$$
 $$\mathbf{J}=\text{ Closure of $C_0^\infty(\Omega;\R^3)\cap \left\{ v:\nabla\cdot v=0 \right\}$ in $H^1(\Omega;\R^3)$},$$
 and
 $$H^1(\Omega, \mathbb{S}^2)=\left\{ \d\in H^1(\Omega,\R^3):\d(x)\in \mathbb{S}^2 \ a.e. \  x\in \Omega \right\}.$$
There is some difference between the weak formulation of non-isothermal systems \eqref{eqn:1.3} or \eqref{eqn:1.4} and that of the isothermal system \eqref{eqn:1.1} or \eqref{eqn:1.2}. 
For example, an important feature of a weak solution to \eqref{eqn:1.1} is the law
of energy dissipation
\begin{equation}
  \frac{d}{dt}\int_{\Omega}\left(|\u|^2+|\nabla\d|^2\right) dx=-2\int_{\Omega}\left( \mu|\nabla\u|^2+|\Delta \d-f_\varepsilon(\d)|^2 \right)dx\leq 0,
  \label{3.5}
\end{equation}
or 
\begin{equation}
  \frac{d}{dt}\int_{\Omega}\left( |\u|^2+|\nabla\d|^2 \right)dx=-2\int_{\Omega}\left( \mu|\nabla\u|^2+|\Delta\d+|\nabla\d|^2\d|^2 \right)dx\leq 0
  \label{3.6}
\end{equation}
for \eqref{eqn:1.2}. 

In contrast with \eqref{3.5} and \eqref{3.6}, we need to include a weak formulation both the \textit{first law of thermodynamics} \eqref{therm1} and the \textit{second law of thermodynamics} \eqref{therm2} into 
\eqref{eqn:1.3} or\eqref{eqn:1.4}. Namely,
the entropy inequality for the temperature equation in  (\ref{eqn:1.3}):
 \begin{eqnarray}
 && \partial_t H(\theta)+\u\cdot \nabla H(\theta)\nonumber\\
 &&\geq-\div(H'(\theta)\q)+H'(\theta)\left( \mu(\theta)|\nabla \u|^2+|\Delta \d-f_\varepsilon(\d)|^2 \right)+H''(\theta)\q\cdot \nabla \theta,
   \label{eqn:RelaxedEntropy}
 \end{eqnarray}
or in \eqref{eqn:1.4}:
\begin{eqnarray}
&&  \partial_t H(\theta)+\u\cdot \nabla H(\theta)\nonumber\\
&&\geq -\div( H'(\theta)\q )+H'(\theta)\left( \mu(\theta)|\nabla \u|^2+|\Delta \d+|\nabla \d|^2 \d|^2 \right)+H''(\theta)\q\cdot \nabla \theta,
  \label{eqn:Entropy}
\end{eqnarray}
 where $H$ is any smooth, non-decreasing and concave function.
More precisely, we have the following weak formulation to the non-isothermal system \eqref{eqn:1.4}.
 \begin{Definition}
   For $0<T< \infty$, a triple $(\u, \d, \theta)$ is a weak solution to  (\ref{eqn:1.4}), (\ref{eqn:Entropy}) if the following properties hold: 
   \begin{itemize} 
  \item [i)]  $\u\in L^\infty([0, T], \H)\cap L^2([0,T], {\bf J})$, 
  $\d\in L^2([0,T], H^1(\Omega,\mathbb S^2))$, $\theta\in L^\infty([0,T], L^1(\Omega)).$
   
 \item [ii)] For any $\varphi\in C_0^\infty(\overline{\Omega}\times [0,T),
 \R^3)$, with $\nabla\cdot\varphi=0$ and
 $\varphi\cdot \nu|_{\partial \Omega}=0$, $\psi_1\in C^\infty_0(\overline{\Omega}\times [0,T),
 \R^3)$, 
 and $\psi_2\in C^\infty(\bar{\Omega}\times [0,T))$ with $\psi_2\geq0$, it holds
\begin{eqnarray}
&&\int_{0}^{T}\int_{\Omega}\left( \u\cdot \partial_t \varphi+ \u\otimes \u:\nabla \varphi\right)\nonumber\\
&&=\int_{0}^{T}\int_{\Omega}(\mu(\theta)\nabla \u-\nabla \d\odot \nabla \d):\nabla \varphi-\int_{\Omega}\u_0\cdot \varphi(\cdot , 0),
       \label{eqn:weaku}
     \end{eqnarray}
     \begin{eqnarray}
      && \int_{0}^{T}\int_{\Omega}(\d\cdot \partial_t \psi_1+\u\otimes \d:\nabla \psi_1)\nonumber\\
       &&=\int_{0}^{T}\int_{\Omega}(\nabla \d:\nabla \psi_1-|\nabla \d|^2 \d\cdot \psi_1)-\int_{\Omega}\d_0\cdot \psi_1(\cdot , 0),
       \label{weakd}
     \end{eqnarray}

     \begin{eqnarray}
      &&\int_{0}^{T}\int_{\Omega}H(\theta)\partial_t \psi_2+\left( H(\theta)\u-H'(\theta)\q \right)\cdot  \nabla \psi_2\nonumber\\
       &&\leq -\int_{0}^{T}\int_{\Omega}\left[H'(\theta)\left( \mu(\theta)|\nabla \u|^2+|\Delta \d+|\nabla \d|^2 \d|^2 \right)-H''(\theta)\q\cdot \nabla \theta\right]\psi_2\nonumber\\
       &&\ \ \ -\int_{\Omega}H(\theta_0)\psi_2{(\cdot , 0)},\label{weakEnt}
     \end{eqnarray}
     for any smooth, non-decreasing and concave function $H$.
\item [iii)]  The following the energy inequality \eqref{therm1}
\begin{equation}
      \int_\Omega\big(\frac12({|\u|^2}+{|\nabla\d|^2})+\theta\big)(\cdot,
      t)\leq \int_{\Omega}\big(\frac12({|\u_0|}+{|\nabla\d_0|^2})+\theta_0\big)
        \label{weaktotal}
     \end{equation}
   holds  for a.e. $t\in [0,T)$.
     \item [iv)] The initial condition $\u(\cdot,0)=\u_0$, $\d(\cdot,0)=\d_0$, $\theta(\cdot,0)=\theta_0$
     holds in the weak sense.
\end{itemize}
 \end{Definition}

Now we state our main result of this paper, which is the following existence theorem 
of global weak solutions to \eqref{eqn:1.4}. 
\begin{Theorem}\label{main}
  For any $T>0,  \u_0\in \mathbf{H}$, $\d_0\in H^1(\Omega, \mathbb{S}^2)$ and $\theta_0\in L^1(\Omega)$, 
  if $ \d_0(\Omega)\subset \mathbb{S}_+^2$ and  $\essinf_\Omega\theta_0>0$, then 
  there exists a global weak solution $(\u, \d, \theta)$ to (\ref{eqn:1.4}), (\ref{eqn:Entropy}), subject to the initial condition $(\u, \d, \theta)=(\u_0, \d_0, \theta_0)$ and
the boundary condition (\ref{eqn:Boundaryud}) and (\ref{eqn:Boundaryq}) such that
\begin{enumerate}
     \item $\u\in L_t^\infty L_x^2\cap L_t^2 H_x^1$,
     \item $\d\in L_t^\infty H_x^1
     (\Omega, \mathbb{S}^2)$, and $\d(x, t)\in \mathbb{S}^2_+$ a.e. in $\Omega\times (0, T)$,
     \item $\theta\in L_t^\infty L_x^1\cap L_t^p W^{1, p}_x$ for $1\leq p<5/4$, $\theta{\geq\essinf_{\Omega}\theta_0}$ a.e. in $\Omega\times (0, T)$.
   \end{enumerate}
  \label{thm:WeakErisksenLesile}
\end{Theorem}

The proof of Theorem \ref{main} is given in the sections below.

\section{Maximum principle with homogeneous Neumann boundary {conditions}}

In this section, we will sketch two a priori estimates for a drifted Ginzburg-Landau heat flow 
under the homogeneous Neumann boundary condition, which is
similar to \cite{LW2} where the {Dirichlet} boundary condition
is considered. 
More precisely,  for $\varepsilon>0$, we consider
\begin{equation}
  \left\{
    \begin{array}{ll}
      \partial_t\d_\varepsilon+\ww\cdot \nabla\d_\varepsilon=\Delta \d_\varepsilon+\frac{1}{\varepsilon^2}\left( 1-|\d_\varepsilon|^2 \right)\d_\varepsilon &\text{ in }\Omega\times(0, T), \\
      \nabla \cdot \ww=0 &\text{ in }\Omega\times(0, T), \\
      \d_\varepsilon(x, 0)=\d_0(x) &\text{ on }\Omega, \\
      \ww=\frac{\partial\d_\varepsilon}{\partial\nu}=0 & \text{ on }\partial\Omega\times(0, T).
    \end{array}
    \right.
    \label{eqn:Neumann}
  \end{equation}
  Then we have
  \begin{Lemma}
    For $0<T\le\infty$, assume $\ww\in L^2([0,T], \mathbf{J})$ and $\d_0\in H^1(\Omega, \mathbb{S}^2)$. Suppose $\d_\varepsilon\in L^2([0,T]; H^1(\Omega, \R^3))$ solves \eqref{eqn:Neumann}. Then
    \begin{equation}
    	|\d_\varepsilon(x, t)|\leq 1 \text{ a.e. } (x,t)\in\Omega\times[0,T].
    	\label{eqn:MAX1}
    \end{equation}
    \label{lemma:MAX1}
  \end{Lemma}
  \begin{proof}
    Set 
    $$v^\varepsilon=(|\d_\varepsilon|^2-1)_+=\begin{cases}
    |\d_\varepsilon|^2-1 & if\ |\d_\varepsilon|\ge 1,\\
    0 & if\ |\d_\varepsilon|<1.
    \end{cases}$$    
  Then $v^\varepsilon$ is a weak solution to
    \begin{equation}
      \left\{
	\begin{array}{ll}
	  \partial_t v^\varepsilon+\ww\cdot \nabla v^\varepsilon=\Delta v^\varepsilon-2\big( |\nabla\d_\varepsilon|^2+\frac{1}{\varepsilon^2}
	  v^\varepsilon|\d_\varepsilon|^2 \big)\leq \Delta v^\varepsilon
	   & \text{in }\Omega\times  (0,T), \\
	  \nabla\cdot \ww=0 & \text{in }\Omega\times(0,T), \\
	  v^\varepsilon(x, 0)=0 & \text{on }\Omega, \\
	  \ww=\frac{\partial v^\varepsilon}{\partial\nu}=0 & \text{on }\partial \Omega\times (0,T).
	\end{array}
	\right.
	\label{eqn:4.1}
      \end{equation}
Multiplying \eqref{eqn:4.1}$_1$ by $v^\varepsilon$ 
and integrating it over $\Omega\times [0,\tau]$ for any $0<\tau\leq T$,
we get
      $$\int_{\Omega}|v^\varepsilon(\tau)|^2+2\int_{0}^{\tau}\int_{\Omega}|\nabla v^\varepsilon|^2\leq -\int_{0}^{\tau}\int_{\Omega}\ww\cdot \nabla((v^\varepsilon)^2)=0. $$
      Thus $v^\varepsilon=0$ a.e. in $\Omega\times[0,T]$ and \eqref{eqn:MAX1} holds.
  \end{proof}

  \begin{Lemma}
For $0<T\le\infty$, assume $\ww\in L^2([0,T]; \mathbf{J})$ and $\d_0\in H^1(\Omega; \mathbb{S}^2)$, with $\d_0(x)\in \mathbb{S}^2_+$ a.e $x\in\Omega$. If $\d_\varepsilon\in L^2([0,T]; H^1(\Omega;\R^3))$ solves \eqref{eqn:Neumann}, then
    \begin{equation}
    	\d_\varepsilon^{3}(x,t)\ge 0 {\text{ a.e.}} \ (x,t)\in\Omega\times[0,T].
    	\label{eqn:MAX2}
    \end{equation}
    \label{lemma:MAX2}
  \end{Lemma}
\begin{proof}
Set $\varphi_\varepsilon(x,t)=\max\{-e^{-\frac{t}{\varepsilon^2}}{\d^{3}}_\varepsilon(x, t),0\}$. Then
\begin{equation}
\left\{
\begin{array}{ll}
\partial_t \varphi_\varepsilon+\ww\cdot \nabla\varphi_\varepsilon-
\Delta \varphi_\varepsilon=\alpha_\varepsilon\varphi_\varepsilon, &\text{ in }\Omega\times(0,T), \\
\nabla\cdot \ww=0, &\text{ in }\Omega\times(0,T), \\
\varphi_\varepsilon(x, 0)=0, & \text{ on }\Omega, \\
\ww=\frac{\partial\varphi_\varepsilon}{\partial\nu}=0, &\text{ on }\partial\Omega\times(0,T),
\end{array}
\right.
\label{eqn:MAX22}
\end{equation}
where
$$\alpha_\varepsilon(x, t)=\frac{1}{\varepsilon^2}(1-|\d_\varepsilon(x,t)|^2)-\frac{1}{\varepsilon^2}\leq 0\ {\text{ a.e. in }}\Omega\times[0,T].$$
%since by Lemma \ref{lemma:MAX1} $|\d_\varepsilon|\leq 1$ a.e. in $\Omega\times[0,T].$
Multiplying \eqref{eqn:MAX22}$_1$ by $\varphi_\varepsilon$
and integrating over $\Omega\times[0, \tau]$ for $0<\tau\leq T$, we obtain
\begin{align*}
  \int_{\Omega}|\varphi_\varepsilon|^2(\tau)+2\int_{0}^{\tau}\int_{\Omega}
  |\nabla \varphi_\varepsilon|^2&=-\int_{0}^{\tau}\int_{\Omega}\ww\cdot \nabla (\varphi_\varepsilon^2)
  +2\int_{0}^{\tau}\int_{\Omega}\alpha_\varepsilon|\varphi_\varepsilon|^2\\
  &=2\int_{0}^{\tau}\int_{\Omega}\alpha_\varepsilon|\varphi_\varepsilon|^2\leq 0.
\end{align*}
Thus $\varphi_\varepsilon=0$ a.e. {in} $\Omega\times[0,T]$ and \eqref{eqn:MAX2} holds.
\end{proof}

%\begin{Remark}The maximum principle results \eqref{eqn:MAX1} and \eqref{eqn:MAX2} are both independent of $\varepsilon$ and $\ww$ and these uniform estimates play an important role in the approximation scheme for weak solutions.\end{Remark}

Finally we need the following minimum principle for the temperature which guarantees 
the positive lower bound of $\theta$.
\begin{Lemma}\label{theta-est}
  For $0<T\le\infty$, assume $\ww\in L^2(0,T; \mathbf{J})$, $\theta_0\in L^1(\Omega)$ with $\essinf_{\Omega}\theta_0>0$, and $\d_\varepsilon\in  L^2([0,T]; H^1
  (\Omega,\mathbb R^3))$. If $\theta_\varepsilon\in L_t^\infty(0,T; L^2(\Omega))\cap L^2(0,T; W^{1,2}(\Omega))$ solves
    \begin{equation}
    \left\{
      \begin{array}{ll}
	\partial_t \theta_\varepsilon+\ww\cdot \nabla \theta_\varepsilon=-\nabla\cdot \q_\varepsilon+\mu(\theta_\varepsilon)|\nabla \ww|^2+|\Delta \d_\varepsilon-\f_\varepsilon(\d_\varepsilon)|^2, & \text{ {\rm in }} \ \ \Omega\times(0,T), \\
	\nabla\cdot \ww=0, & \text{ {\rm in} }\ \ \Omega\times(0,T), \\
	\theta_\varepsilon(x, 0)=\theta_0(x), & \text{ {\rm on } }\ \Omega, \\
	\ww=\q_\varepsilon\cdot\nu=0, &\text{ {\rm on } } \ \partial \Omega\times(0,T),
      \end{array}
      \right.
      \label{eqn:MAXT1}
    \end{equation}
    where $\q_\varepsilon=-k(\theta_\varepsilon)\nabla\theta_\varepsilon-h(\theta_\varepsilon)(\nabla \theta_\varepsilon\cdot \d_\varepsilon)\d_\varepsilon$,
   then
   \begin{equation}
     \theta_\varepsilon(x,t)\geq {{\essinf_{\Omega}} \theta_0} {\text{ a.e. in } \Omega\times[0,T].}
     \label{eqn:MAXT}
   \end{equation}
  \label{lemma:MAXT}
\end{Lemma}
\begin{proof}
  Let $\theta_{\varepsilon}^-=\max\left\{{\essinf_{\Omega}\theta_0} -\theta_\varepsilon,0  \right\}$.
  Then by direct computation, {\eqref{eqn:MAXT1}} implies that
 \begin{equation}
    \left\{
      \begin{array}{ll}
	\partial_t \theta_{\varepsilon}^-+\ww\cdot \nabla \theta_{\varepsilon}^-\leq{-\nabla\cdot\q_\varepsilon^-}, & \text{ in }\Omega\times{(0,T)}, \\
	\nabla\cdot \ww=0, & \text{ in }\Omega\times{(0,T)}, \\
	\theta_{\varepsilon}^-(x, 0)=0, & \text{ on }\Omega, \\
	\ww=\q_\varepsilon^-\cdot\nu=0, &\text{ on }\partial \Omega\times{(0,T)},
      \end{array}
      \right.
      \label{eqn:MAXTN}
    \end{equation}
    where $\q_\varepsilon^-=-k(\theta_\varepsilon)\nabla \theta^-_\varepsilon-h(\theta_\varepsilon)(\nabla \theta^-_\varepsilon\cdot \d_\varepsilon)\d_\varepsilon$. \\
    Multiplying \eqref{eqn:MAXTN}$_1$ by $\theta_{\varepsilon}^-$ 
    and integrating over $\Omega\times[0,\tau]$ for $0<\tau\leq T$, we obtain
    \begin{equation*}
      \int_{\Omega}|\theta_{\varepsilon}^-|^2(\tau)+2\int_{0}^{\tau}\int_{\Omega}\underline{k}\left(|\nabla \theta_{\varepsilon}^-|^2+|\nabla\theta_{\varepsilon}^-\cdot\d_\varepsilon|^2\right)\leq 0.  \end{equation*}
    Therefore $\theta_{\varepsilon}^-=0$ a.e. {in} $\Omega\times[0,T]$,
    which yields \eqref{eqn:MAXT}.
\end{proof}

     \section{Existence of weak solutions to (\ref{eqn:modified-system})}

     In this section we will sketch the construction of weak solutions to (\ref{eqn:modified-system})
     by the Faedo-Galerkin method, which is similar to that by \cite{FFRS} and \cite{LL1}.
     To simplify the presentation, we only consider the case $\varepsilon=1$ and construct
     a weak solution of the following system:
    \begin{equation}
      \left\{
	\begin{array}{l}
	  \partial_t\u+\u\cdot \nabla \u+\nabla P=\div\left( \mu(\theta)\nabla \u-\nabla\d\odot \nabla \d \right), \\
	  \nabla \cdot \u=0, \\
	  \partial_t\d+\u\cdot \nabla \d=\Delta \d-\f(\d), \\
	 \partial_t\theta+\u\cdot \nabla \theta=-\div \q+\mu(\theta)|\nabla \u|^2+|\Delta \d-\f(\d)|^2,
	\end{array}
	\right.
	\label{eqn:modified-system}
      \end{equation}
      where {$\f(\d)=\partial_{\d} F(\d)=(|\d|^2-1)\d$.}
      
      Let $\left\{ \varphi_i \right\}_{i=1}^{\infty}$ be an orthonormal basis of $\mathbf{H}$ 
      formed by eigenfunctions of the Stokes operator on $\Omega$ with zero Dirichlet boundary condition, i.e.,
      \begin{equation*}
	\left\{
	  \begin{array}{ll}
	    -\Delta \varphi_i+\nabla P_i=\lambda_i \varphi_i & \text{ in }\Omega, \\
	    \nabla\cdot\varphi_i =0 & \text{ in }\Omega,\\
	    \varphi_i=0 & \text{ on }\partial \Omega,
	  \end{array}
	  \right.
	  \label{}
	\end{equation*}
	for $i=1, 2, \cdots$, and $0<\lambda_1\leq\lambda_2\leq\cdots\leq\lambda_n\leq \cdots$, with $\lambda_n\rightarrow\infty$.
	
	%Also let $\left\{ \psi_i \right\}_{i=1}^{\infty}$ be the eigenfunctions of the Laplace operator
	%$-\Delta$ on $\Omega$ with zero Neumann boundary condition, corresponding to
	%eigenvalues $0=\mu_1<\mu_2\le\cdots\le\mu_n\le\cdots$ with
	%$\mu_n\rightarrow\infty$, which form an orthonormal basis of $L^2(\Omega)$.
	
	Let $\PP_m:\mathbf{H}\rightarrow \mathbf{H}_m=span\left\{ \varphi_1, \varphi_2, \cdots, \varphi_m \right\}$ be the orthogonal projection operator. Consider 
	\begin{equation} \label{eqn:um}
	  \left\{
	    \begin{array}{l}
	      \partial_t \u_m=\PP_m\big[-\u_m\cdot \nabla \u_m+\div\left(\mu(\theta_m)\nabla \u_m-\nabla \d_m\odot \nabla \d_m\right) \big], \\
	      \u_m(\cdot, t)\in {\bf H}_m, \ \ \forall t\in [0,T), \\
	      \u_m(x,0)=\PP_m(\u_0)(x),  \ \ \forall x\in\Omega,
	    \end{array}
	    \right.
	   \end{equation}
	  \begin{equation}
	    \left\{
	      \begin{array}{l}
		\partial_t \d_m+\u_m\cdot \nabla \d_m=\Delta \d_m-f(\d_m), \\
		\d_m(x, 0)=\d_0(x)\  \ \forall x\in\Omega, \\
	\frac{\partial\d_m}{\partial\nu}=0\ \ {\text{ on }}\ \partial \Omega,
	      \end{array}
	      \right.
	      \label{eqn:dm}
	    \end{equation}
	   % Let $\mathbb{Q}_m:L^2(\Omega)\to {\bf G}_m=span\{\psi_1,\cdots,\psi_m\}$ be the orthogonal
	    %projection operator, we also consider the following approximation system
	    \begin{equation}
	      \left\{
		\begin{array}{l}
		  \partial_t \theta_m+\u_m\cdot\nabla \theta_m=
		  \div \big(k(\theta_m)\nabla\theta_m+h(\theta_m)(\nabla\theta_m\cdot\d_m)\d_m\big)\\
		  \qquad\qquad\qquad\qquad\ \ \ \ +\mu(\theta_m)|\nabla \u_m|^2+|\Delta \d_m-\f(\d_m)|^2, \\
		   \theta_m(x, 0)=\theta_0(x)\ \ \forall x\in\Omega,\\
		   \frac{\partial\theta_m}{\partial\nu}=0 \ \ {\text{ on }}\ \partial\Omega.
		  \end{array}
		\right.
		\label{eqn:thetam}
	      \end{equation}
	      Since $\u_m(\cdot, t )\in \mathbf{H}_m$, 
	      we can write
	      \begin{equation*}
		\u_m(x,t)=\sum_{i=1}^{m}g_m^{(i)}(t)\varphi_i(x),
		\end{equation*}
	      so that \eqref{eqn:um}  becomes the following system of ODEs:
	      \begin{equation}\label{ODE}
		\frac{d}{dt}g_m^{(i)}(t)=A^{(i)}_{jk}g_m^{(j)}(t)g^{(k)}_m(t)+B_{mj}^{(i)}(t) g_m^{(j)}(t)+C_m^{(i)}(t),		          \end{equation}
		subject to the initial condition
		\begin{equation}\label{IC}
		g^{(i)}_m(0)=\int_\Omega\langle\u_0, \varphi_i\rangle,
	      \end{equation}
	      for $1\le i\le m$, where
	      \begin{align*}
		A^{(i)}_{jk}&=-\int_{\Omega}\langle\varphi_j\cdot \nabla \varphi_k,  \varphi_i\rangle, \\
		B^{(i)}_{mj}(t)&=-\int_{\Omega}\langle\mu(\u_m)\nabla \varphi_j,\nabla{\color{red}\varphi_i}\rangle,\\
		C_m^{(i)}(t)&=\int_{\Omega}(\nabla \d_m\odot \nabla \d_m):\nabla \varphi_i,
	       \end{align*}
	      for $1\le j, k\le m$.
	      
	      For $T_0>0$ and $M>0$ to be chosen later, suppose 
	      $\big(g_m^{(1)},\cdots, g_m^{(m)}\big)\in C^1([0,T_0])$ and
	      \begin{equation}\label{bound1}
	      \sup_{0\le t\le T_0}\sum_{i=1}^{m}|g_m^{(i)}(t)|^2\leq M^2.
	      \end{equation}
	      Since $\partial_t\u_m, \nabla^2 \u_m\in C^0(\Omega\times [0,T_0])$,
	      the standard theory of parabolic equations  implies that    
	      there exists a strong solution $\d_m$ to \eqref{eqn:dm} such that
	      for any $\delta>0$, $\partial_t\d_m,\nabla^2\d_m\in L^p(\Omega\times [\delta,T_0])$
	      for any $1\le p<\infty$ (see \cite{LSN}). Next we can solve  \eqref{eqn:thetam} to
	      obtain a nonnegative, strong solution $\theta_m$. In fact, observe that
	      $$k(\theta_m)\nabla\theta_m+h(\theta_m)(\nabla\theta_m\cdot\d_m)\d_m
	      =D(\theta_m)\nabla\theta_m,$$
	      where $(D_{ij}(\theta_m))=(k(\theta_m)\delta_{ij}+h(\theta_m)\d_m^i\d_m^j)$ is uniformly elliptic,
	      and $\mu(\theta_m)|\nabla\u_m|^2+|\Delta\d_m-\f(\d_m)|^2\in L^p(\Omega\times [\delta,T_0])$
	      holds for any $1<p<\infty$ and $\delta>0$.  Thus by the standard theory of parabolic equations,
	      we can first obtain a unique weak solution $\theta_m$ to \eqref{eqn:dm} such that 
	      $\theta_m\in C^\alpha(\overline\Omega\times [\delta, T_0])$ for some $\alpha\in (0,1)$. This
	      yields that the coefficient matrix $D(\theta_m)\in C(\overline\Omega\times [\delta, T_0])$
	      and hence by the regularity theory of parabolic equations we conclude that 
	      $\nabla\theta_m\in L^p(\Omega\times [\delta, T_0])$ for any $1<p<\infty$ and $\delta>0$.
	      Now we see that $\theta_m$ satisfies
	      $$
	      \partial_t\theta_m-D_{ij}(\theta_m)\frac{\partial^2\theta_m}{\partial x_i\partial x_j}
	      =D_{ij}'(\theta_m) \frac{\partial\theta_m}{\partial x_i} \frac{\partial\theta_m}{\partial x_j}
	      +\mu(\theta_m)|\nabla\u_m|^2+|\Delta\d_m-\f(\d_m)|^2,
	      $$
	      where $|D_{ij}'(\theta_m)|\le |h'(\theta_m)|+|k'(\theta_m)|$ is bounded, since $h$ and $k$
	      are Lipschitz continuous. Hence by the $W^{2,1}_p$-theory of parabolic
	      equations, $\partial_t\theta_m, \nabla^2\theta_m\in L^p(\Omega\times [\delta, T_0])$ 
	      for any $1<p<\infty$ and $\delta>0$.
	 
	      To solve \eqref{ODE} and \eqref{IC}, we need some apriori estimates.
	      Taking the $L^2$ inner product of \eqref{eqn:dm} with ${-\Delta \d_m+\f(\d_m)}$ yields
	      \begin{align*}
		\frac{d}{dt}\int_{\Omega}|\nabla \d_m|^2+2F(\d_m)
		&=-2\int_{\Omega}|\Delta \d_m-\f(\d_m)|^2
		+2\int_{\Omega}(\u_m\cdot \nabla \d_m)\cdot (\Delta \d_m-\f(\d_m))\\
		&\leq -\int_{\Omega}|\Delta \d_m-\f(\d_m)|^2+\int_{\Omega}|\u_m\cdot \nabla \d_m|^2, 
		\quad t\in [0,T_0].
	      \end{align*}
	      It follows from \eqref{bound1} that
	      \begin{equation*}
		\left\|\u_m\right\|_{L^\infty(\Omega\times [0,T_0])}
		\leq M\cdot \max_{1\leq i\leq m}\left\|\varphi_i\right\|_{L^\infty(\Omega)}\leq C_m M.
		\label{}
	      \end{equation*}
Therefore we get
$$\frac{d}{dt}\int_{\Omega} (|\nabla \d_m|^2+2 F(\d_m)) +\int_{\Omega}|\Delta \d_m-\f(\d_m)|^2\leq C_m^2 M^2 \int_{\Omega}|\nabla \d_m|^2.$$
This, combined with Gronwall's inequality and $F(\d_0)=0$, implies
$$\sup_{0\le t\le T_0}\int_{\Omega}(|\nabla \d_m|^2+F(\d_m))
+\int_0^{T_0}\int_{\Omega}|\Delta \d_m-\f(\d_m)|^2
\leq e^{C_m^2 M^2 T_0}\int_{\Omega}|\nabla \d_0|^2,$$ 
so that
\begin{equation*}
  \sup_{0\le t\le T_0}\max_{1\le i,j\le m}\big(|B_{mj}^{(i)}(t)|+|C_m^{(i)}(t)|\big)\leq C_0(m,M).
  \label{}
\end{equation*}
Thus we can solve  \eqref{ODE}
and \eqref{IC} to obtain a unique solution $(\tilde{g}_m^{(1)}(t),\cdots, \tilde{g}_m^{(m)}(t))\in C^1([0,T_0])$ 
such that for all $t\in [0,T_0]$
\begin{equation}
\sum_{i=1}^{m}|\tilde{g}_m^{(i)}(t)|^2
\leq \sum_{i=1}^{m}|{g}_m^{(i)}(0)|^2 +C(m, M, \underline{\mu}, \overline{\mu},\underline{k}, \overline{k}) t^2.
  \label{eqn:ODEestimate}
\end{equation}
Choose $M=2+2\sum_{i=1}^{m}|g_m^{(i)}(0)|^2$ and $T_0>0$ so small that the right-hand side of \eqref{eqn:ODEestimate} is less than $M^2$ for all $t\in [0,T_0]$. Set $\tilde{\u}_m:\Omega\times [0,T_0]\to \mathbb R^3$ by
$$\tilde{\u}_m(x,t)=\sum_{i=1}^{m}\tilde{g}_m^{(i)}(t)\varphi_i(x).$$
Then $\mathcal{L}(\u_m)=\tilde{\u}_m$ defines
a map from $\mathbf{U}(T_0)$ to $\mathbf{U}(T_0)$, 
where 
\begin{eqnarray*}
\mathbf{U}(T_0)=\Big\{ \u_m(x,t)=
\sum_{i=1}^{m}g_m^{(i)}(t)\varphi_i(x):\max_{t\in [0,T_0]}\sum_{i=1}^{m}|g_m^{(i)}(t)|^2\leq M^2, 
\quad \u_m(0)=\mathbb P_m\u_0\Big\}.
\end{eqnarray*}
Since $\mathbf{U}(T_0)$ is a closed, convex subset of $H_0^1(\Omega)$ and $\mathcal{L}$ is a compact operator, it follows from the Leray-Schauder theorem that $\mathcal{L}$ has a fixed point
$\u_m\in \mathbf{U}(T_0)$ for the approximation system \eqref{eqn:um}, 
and a classical solution $\d_m$ to \eqref{eqn:dm} and $\theta_m$ to \eqref{eqn:thetam} on $\Omega\times [0,T_0]$, see \cite{Evans}. 

Next, we will {establish} a priori estimates and show that the solution can be extended to $[0,T]$. 
To do it, taking the $L^2$ inner product of \eqref{eqn:um} and \eqref{eqn:dm} by $\u_m$ and $-\Delta \d_m+\f(\d_m)$ respectively, and adding together these two equations, we get that for $t\in [0, T_0]$, 
\begin{equation}
\frac{d}{dt}\int_{\Omega}(|\u_m|^2+|\nabla \d_m|^2 +{2}F(\d_m))
  +2\int_{\Omega}\mu(\theta_m)|\nabla \u_m|^2+|\Delta \d_m-\f(\d_m)|^2=0,
  \label{eqn:approximateInternal}
\end{equation}
where we use the identities
$$\int_{\Omega}\u_m\cdot \div(\nabla \d_m\odot \nabla \d_m)
=\int_\Omega (\u_m\cdot\nabla\d_m)\cdot\Delta\d_m,$$
$$\int_{\Omega}(\u_m\cdot \nabla \d_m)\cdot \f(\d_m)=\int_{\Omega} \u_m \cdot  \nabla F(\d_m)=0.$$
We can derive from \eqref{eqn:approximateInternal} that
\begin{eqnarray}
  &&\sup_{0\leq t\leq T_0}\int_{\Omega}(|\u_m|^2+|\nabla \d_m|^2+{2}F(\d_m))
  +2\int_{0}^{T_0}\int_{\Omega}\mu(\theta_m)|\nabla \u_m|^2+|\Delta \d_m-\f(\d_m)|^2 \nonumber\\
  &&\leq \int_{\Omega}(|\u_0|^2+|\nabla \d_0|^2).
  \label{eqn:approximateInternal2}
\end{eqnarray}
Lemma \ref{lemma:MAX1} implies that 
$|\d_m|\leq 1$ and $|\f(\d_m)|\le 1$ in $\Omega\times[0,T_0]$, so that
\begin{align*}
  \int_{0}^{T_0}\int_{\Omega}|\Delta \d_m|^2
  \le 2\int_{0}^{T_0}\int_{\Omega}(1+|\Delta \d_m-\f(\d_m)|^2).
  \end{align*}
Hence \eqref{eqn:approximateInternal2} yields thqat
\begin{eqnarray}
 && \sup_{0\leq t\leq T_0}\int_{\Omega}(|\u_m|^2+|\nabla \d_m|^2)
 +\int_{0}^{T_0}\int_{\Omega}(\underline{\mu}|\nabla \u_m|^2+|\Delta \d_m|^2)\nonumber\\
 && \leq \int_{\Omega}(|\u_0|^2+|\nabla \d_0|^2)+CT_0|\Omega|.
  \label{eqn:approximateInternal3}
\end{eqnarray}
While the integration of \eqref{eqn:thetam} over $\Omega$ yields
\begin{equation}\label{eqn:thetam-id}
\frac{d}{dt}\int_\Omega \theta_m =\int_\Omega (\mu(\theta_m) |\nabla\u_m|^2+|\Delta \d_m-\f(\d_m)|^2).
\end{equation}
Adding \eqref{eqn:approximateInternal} together with \eqref{eqn:thetam-id} and integrating 
over $[0,T_0]$, we obtain 
 \begin{equation}\label{bound4}
\sup_{0\le t\le T_0}
\int_\Omega(|\u_m|^2+|\nabla \d_m|^2+\theta_m) 
\le \int_\Omega (|\u_0|^2+|\nabla\d_0|^2+\theta_0).
\end{equation}

Next by choosing $H(\theta)=(1+\theta)^\alpha$, $\alpha\in (0,1)$, and multiplying
the equation \eqref{eqn:thetam} by $H'(\theta_m)=\alpha (1+\theta_m)^{\alpha-1}$,
we get
\begin{eqnarray}
  &&\partial_t(1+\theta_m)^\alpha+\u_m\cdot \nabla(1+\theta_m)^\alpha\nonumber\\
  &&=-\div\left( \alpha(1+\theta_m)^{\alpha-1}\q_m \right)
  +\alpha(1+\theta_m)^{\alpha-1}\left( \mu(\theta_m)|\nabla \u_m|^2+|\Delta \d_m-\f(\d_m)|^2 \right)\nonumber\\
  &&\quad+\alpha(\alpha-1)(1+\theta_m)^{\alpha-2}\q_m\cdot \nabla \theta_m,
  \label{eqn:approximateEntropy}
\end{eqnarray}
where $\q_m=-h(\theta_m)\nabla\theta_m-k(\theta_m)(\nabla\theta_m\cdot\d_m)\d_m$. 

Integrating \eqref{eqn:approximateEntropy} over ${\Omega\times[0,T_0]}$ yields
\begin{equation}
\int_{0}^{T_0}\int_{\Omega}\alpha(\alpha-1)(1+\theta_m)^{\alpha-2}\q_m\cdot \nabla\theta_m
\le \int_{\Omega\times\{T_0\}} (1+\theta_m)^\alpha-\int_{\Omega} (1+\theta_0)^\alpha.
\end{equation}
Notice that 
\begin{align*}
  &\int_{0}^{T_0}\int_{\Omega}\alpha(\alpha-1)(1+\theta_m)^{\alpha-2}\q_m\cdot \nabla\theta_m\\
  &=\alpha(1-\alpha)\int_{0}^{T_0}\int_{\Omega}(1+\theta_m)^{\alpha-2}(k(\theta_m)|\nabla\theta_m|^2+h(\theta_m)(\nabla\theta_m\cdot \d_m)^2)\\
  &\ge \alpha(1-\alpha)\underline{k}\int_0^{T_0}\int_\Omega (1+\theta_m)^{\alpha-2}|\nabla\theta_m|^2\\
  &\ge\frac{4\alpha(1-\alpha)\underline{k}}{\alpha^2}\int_0^{T_0}\int_\Omega |\nabla\theta_m^{\frac{\alpha}2}|^2.
  \end{align*}
 Thus  we obtain that
 \begin{eqnarray}
 \int_{0}^{T_0}\int_{\Omega}\big|\nabla \theta_m^{\frac{\alpha}{2}}\big|^2
 &&\le C(\alpha,\underline{k}) \int_{\Omega\times\{T_0\}} (1+\theta_m)^\alpha\nonumber\\
 &&\le  C(\alpha,\underline{k}, \Omega) \big(\int_{\Omega\times\{T_0\}} (1+\theta_m)\big)^\alpha\nonumber\\
 && \leq C(\alpha, \underline{k}, \Omega) \Big(1+\int_\Omega (|\u_0|^2+|\nabla \d_0|^2+\theta_0)\Big)^\alpha.
  \label{eqn:nablathetam}
\end{eqnarray}
With \eqref{bound4} and \eqref{eqn:nablathetam}, we can apply an interpolation 
argument, similar to (4.13) in \cite{FFRS}, to conclude that $\theta_m\in L^q(\Omega\times [0,T_0])$
for any $1\le q<\frac53$, and 
\begin{equation}\label{bound6}
  \left\|\theta_m\right\|_{L^q({\Omega\times[0,T]})}\leq C\big(q, \underline{k}, \|\u_0\|_{L^2(\Omega)},
  \|\nabla\d_0\|_{L^2(\Omega)}, \|\theta_0\|_{L^1(\Omega)}\big). 
\end{equation}
This, together with \eqref{eqn:nablathetam} and H\"{o}lder's inequality: 
\begin{equation*}
  \int_{{\Omega\times[0,T_0]}}|\nabla \theta_m|^p\leq \big( \int_{{\Omega\times[0,T_0]}}|\nabla \theta_m|^2 \theta_m^{\alpha-2} \big)^{\frac{p}{2}}
  \big( \int_{{\Omega\times[0,T_0]}}\theta_m^{(2-\alpha)\frac{p}{2-p}} \big)^{\frac{2-p}{2}},
  \end{equation*}
  for $\alpha\in (0,1)$ and $1\le p<2$,  implies that 
  \begin{equation}
  \big\|\nabla\theta_m\big\|_{L^{{\color{red} p}}(\Omega\times [0,T_0])}
  \le C\big({\color{red} p},\underline{k}, \|\u_0\|_{L^2(\Omega)},
  \|\nabla\d_0\|_{L^2(\Omega)}, \|\theta_0\|_{L^1(\Omega)}\big)
  \label{eqn:Lpthetam}
\end{equation}
holds for all $p\in [1,5/4)$.

Plugging the estimates \eqref{eqn:approximateInternal3}, \eqref{bound4}, \eqref{bound6},
and \eqref{eqn:Lpthetam} into the system \eqref{eqn:um}, \eqref{eqn:dm}, and \eqref{eqn:thetam},
we conclude that 
\begin{equation}
  \sup_m\big\{\left\|\partial_t \u_m\right\|_{L^{\frac{4}{3}}(0,T_0;H^{-1}(\Omega))}
  +\left\|\partial_t \d_m\right\|_{L^{\frac{4}{3}}(0,T_0;L^2(\Omega))}
  +\left\|\partial_t\theta_m\right\|_{L^2(0,T_0; W^{-1,4}(\Omega)}\big\}\leq C.
  \label{eqn:udmt}
\end{equation}
Therefore, by setting $\big(\u_m(\cdot, T_0), \d_m(\cdot, T_0),\theta_m(\cdot, T_0)\big)$
as then initial data and repeating the same argument,  we can extend the solution 
to the interval  $[0,2T_0]$ and eventually obtain a solution 
$(\u_m, \d_m, \theta_m)$ to the  system \eqref{eqn:um}, \eqref{eqn:dm}, \eqref{eqn:thetam}
in $[0,T]$ such that  the estimates \eqref{eqn:approximateInternal3}, \eqref{bound4}, \eqref{bound6}, \eqref{eqn:Lpthetam}, and \eqref{eqn:udmt} hold with $T_0$ replaced by $T$.

The existence of a weak  solution to the original system \eqref{eqn:modified-system} will be obtained
by passing to the limit of $(\u_m, \d_m, \theta_m)$ as $m\to\infty$.
In fact, by {Aubin-Lions'} compactness lemma \cite{Simon1996}, we know that
 there exists $\u\in L_t^\infty L_x^2\cap L_t^2 H_x^1(\Omega\times [0,T])$,
 $\d\in L_t^\infty H_x^1\cap L^2_tH^2_x(\Omega\times [0,T])$,
 and a nonnegative $\theta\in L^\infty_tL^1_x\cap L^p_tW^{1,p}_x(\Omega\times [0,T])$,
 for $1<p<\frac54$, such that, after passing to a subsequence,
\begin{equation*}
  \left\{
    \begin{array}{ll}
      \u_m\rightarrow \u & \text{ in }L^2({\Omega\times[0,T])},\\
      (\d_m,\nabla\d_m)\rightarrow (\d,\nabla\d) & \text{ in }L^2({\Omega\times[0,T])},\\
      \theta_m\rightarrow\theta \ & a.e. \ and \text{ in } L^{p_1}(\Omega\times [0,T]), \ \forall 1<p_1<\frac53,\\
      \nabla \u_m\rightharpoonup \nabla \u & \text{ in } L^2({\Omega\times[0,T])},\\
      \nabla^2 \d_m\rightharpoonup \nabla^2 \d & \text{ in } L^2({\Omega\times[0,T])},\\
       \nabla\theta_m\rightharpoonup \nabla\theta & \text{ in } L^{p_2}(\Omega\times [0,T]),
        \ \forall 1<p_2<\frac54.
    \end{array}
    \right.
    \label{}
  \end{equation*}
  Since $\mu\in C([0,\infty))$ is bounded, we have that
  $$\mu(\theta_m)\rightarrow\mu(\theta) \ \ {\rm{in}}\ \ L^p(\Omega\times [0,T]), \ \forall 1\le p<\infty,$$
and
  $$
  \mu(\theta_m)\nabla\u_m\rightharpoonup \mu(\theta)\nabla \u 
  \ \ {\rm{in}}\ \ L^2(\Omega\times [0,T]).
  $$
After passing $m\to\infty$ in the the equations \eqref{eqn:um} and \eqref{eqn:dm}, 
we see that $(\u,\d, \theta)$ satisfies the equations $\eqref{eqn:modified-system}_1$, 
 $\eqref{eqn:modified-system}_2$, and $\eqref{eqn:modified-system}_3$ in the weak sense. 
 
Next we want to verify that $\theta$ satisfies 
\begin{eqnarray}\label{entropy1}
&&\int_0^T\int_\Omega \big(H(\theta)\partial_t\psi+(H(\theta)\u-H'(\theta)\q)\cdot\nabla\psi\big)\nonumber\\
&&\le -\int_0^T\int_\Omega\big[H'(\theta)(\mu(\theta)|\nabla\u|^2+|\Delta\d-\f(\d)|^2)-H''(\theta)\q\cdot\nabla\theta\big]\psi\nonumber\\
&&\ \ \ -\int_\Omega H(\theta_0)\psi(\cdot,0)
\end{eqnarray}
holds for any smooth, non-decreasing and concave function $H$, and $\psi\in C^\infty_0(\overline\Omega\times [0,T))$ with $\psi\ge 0$.  Here $\q=-k(\theta)\nabla\theta-h(\theta)(\nabla\theta\cdot \d)\d.$
Observe that by choosing $H(t)=t$, \eqref{entropy1} yields that $\theta$ solves $\eqref{eqn:modified-system}_4$ in the weak sense, namely, 
\begin{eqnarray}\label{theta1}
&&\int_0^T\int_\Omega \big(\theta\partial_t\psi+(\theta\u-\q)\cdot\nabla\psi\big)\nonumber\\
&&\le -\int_0^T\int_\Omega(\mu(\theta)|\nabla\u|^2+|\Delta\d-\f(\d)|^2)\psi
-\int_\Omega \theta_0\psi(\cdot,0).
\end{eqnarray}

In order to show \eqref{entropy1},   first observe that multiplying the equation 
\eqref{eqn:thetam} by $H'(\theta_m)\psi$, integrating over $\Omega\times [0,T]$,
and employing the regularity of $\theta_m, \u_m, \d_m$ implies
\begin{eqnarray}\label{entropy2}
&&\int_0^T\int_\Omega \big(H(\theta_m)\partial_t\psi+(H(\theta_m)\u_m-H'(\theta_m)\q_m)\cdot\nabla\psi\big)\nonumber\\
&&= -\int_0^T\int_\Omega\big[H'(\theta_m)(\mu(\theta_m)|\nabla\u_m|^2+|\Delta\d_m-\f(\d_m)|^2)-H''(\theta_m)\q_m\cdot\nabla\theta_m\big]\psi\nonumber\\
&&\ \ \ -\int_\Omega H(\theta_0)\psi(\cdot,0),
\end{eqnarray}
where $\q_m=-k(\theta_m)\nabla\theta_m-h(\theta_m)(\nabla\theta_m\cdot \d_m)\d_m.$

{It follows from Lemma \ref{theta-est}  that $\theta_m\geq \essinf_{\Omega} \theta_0$ a.e.. Without loss of generality, we assume $H(0)=0$ so that $H(\theta_m)\geq H(\essinf_\Omega \theta_0)\geq 0$ since $H$ is nondecreasing. From $H''\leq0$, we conclude that $0\leq H'(\theta_m)\leq H'(\essinf_{\Omega} \theta_0)$. }
 From the concavity of $H$, we have
$$\frac{1}{|\Omega|}\int_\Omega H(\theta_m)\le H(\frac{1}{|\Omega|}\int_\Omega \theta_m)$$
so that 
$$\{H(\theta_m)\} \ \mbox{is bounded in} \ \ L^\infty_t L^1_x
\cap L^p_t W^{1,p}_x(\Omega\times [0,T]), \ \forall 1<p<\frac54.
$$
This, combined with the 
bounds on $\theta_m, \u_m, \d_m$ and \eqref{entropy2}, implies
that 
\begin{eqnarray*}
&&\int_0^T\int_\Omega H''(\theta_m) \q_m\cdot\nabla\theta_m\psi\\
&&=\int_0^T\int_\Omega (|\sqrt{-H''(\theta_m) k(\theta_m)\psi}\nabla \theta_m|^2
+|\sqrt{-H''(\theta_m) h(\theta_m)\psi}(\nabla \theta_m\cdot\d_m)|^2)
\end{eqnarray*}
is uniformly bounded. 
{For any fixed $l\in\mathbb{N}^+$, since}
$$\sqrt{\min\{-H''(\theta_m), l\} k(\theta_m)\psi}\nabla \theta_m\rightharpoonup
\sqrt{\min\{-H''(\theta),l\} k(\theta)\psi}\nabla \theta, $$
and
$$
\sqrt{\min\{-H''(\theta_m),l\} h(\theta_m)\psi}(\nabla \theta_m\cdot\d_m)\rightharpoonup
\sqrt{\min\{-H''(\theta),l\} h(\theta)\psi}(\nabla \theta\cdot\d)
$$
in $L^p(\Omega\times [0,T]$ for $1<p<\frac54$,  we have by the lower semicontinuity that
\begin{equation}\label{conv}
%\color{red}
\begin{split}
\int_0^T\int_\Omega \min\{-H''(\theta),l\} \q\cdot\nabla\theta\psi
&\le\liminf_{m\to \infty}\int_0^T\int_\Omega \min\{-H''(\theta_m),l\} \q_m\cdot\nabla\theta_m\psi\\
&\leq \liminf_{m\to \infty}\int_0^T\int_\Omega -H''(\theta_m)\q_m\cdot\nabla\theta_m\psi.
\end{split}
\end{equation}
{This, after sending $l\to\infty$, yields 
	\begin{equation}\label{conv1}
		\int_0^T\int_{\Omega}-H''(\theta)\q\cdot\nabla\theta \psi\leq \liminf_{m\to \infty}\int_0^T\int_\Omega -H''(\theta_m)\q_m\cdot\nabla\theta_m\psi.
	\end{equation}
}
It follows from the lower semicontinuity again that
\begin{eqnarray}\label{conv2}
&&\int_0^T\int_\Omega\big[H'(\theta)(\mu(\theta)|\nabla\u|^2+|\Delta\d-\f(\d)|^2)\psi\nonumber\\
&&\le\liminf_{m\to\infty}
\int_0^T\int_\Omega\big[H'(\theta_m)(\mu(\theta_m)|\nabla\u_m|^2+|\Delta\d_m-\f(\d_m)|^2)\psi.
\end{eqnarray}
On the other hand, since 
$$H(\theta_m)\to H(\theta), \  H(\theta_m)\u_m\to H(\theta)\u \ \ {\rm{in}}\ \ L^1(\Omega\times [0,T]),$$
and
$$H'(\theta_m)\q_m\rightharpoonup H'(\theta) \q \ \ {\rm{in}}\ \ L^1(\Omega\times [0,T]),$$
we have
\begin{eqnarray}\label{conv3}
&&\int_0^T\int_\Omega \big(H(\theta)\partial_t\psi+(H(\theta)\u-H'(\theta)\q)\cdot\nabla\psi\big)\nonumber\\
&&=\lim_{m\to\infty}
\int_0^T\int_\Omega \big(H(\theta_m)\partial_t\psi+(H(\theta_m)\u_m-H'(\theta_m)\q_m)\cdot\nabla\psi\big).
\end{eqnarray}
Therefore \eqref{entropy1} follows by passing $m\to\infty$ in \eqref{entropy2} and applying
\eqref{conv1}, \eqref{conv2}, and \eqref{conv3}.  This completes the construction of a global weak solution
to \eqref{eqn:modified-system}. \qed

\section{Convergence and existence of global weak solutions of \eqref{eqn:1.4}}

In this section, we will apply Lemma \ref{lemma:MAX1}, Lemma \ref{lemma:MAX2}, and Lemma \ref{lemma:MAXT} to analyze the convergence of a sequence of weak solutions 
$(\u_\varepsilon, \d_\varepsilon, \theta_\varepsilon)$ to the Ginzburg-Landau approximate
system \eqref{eqn:1.3} constructed in the previous section,  as $\varepsilon\to 0$,
and obtain  a global weak solution $(\u,\d,\theta)$ to  \eqref{eqn:1.4}. 

Here we will employ the pre-compactness theorem by Lin-Wang \cite{LW2} on
approximated harmonic maps to show that $\d_\varepsilon\to\d$ in 
$L^2([0,T], H^1(\Omega))$ as $\varepsilon\to 0$.

\begin{proof}[Proof of Theorem \ref{thm:WeakErisksenLesile}]

Let $(\u_\varepsilon, \d_\varepsilon, \theta_\varepsilon)$ be the weak solutions to the Ginzburg-Landau
approximate system \eqref{eqn:1.3}, under the boundary condition (\ref{eqn:Boundaryud}), (\ref{eqn:Boundaryq}),  obtained from Section 5. Then  there exist $C_1, C_2>0$ depending only
on $\u_0$, $\d_0$, and $\theta_0$ such that 
\begin{eqnarray*}
 &&\sup_{\varepsilon}\Big\{\|\u_{\varepsilon}\|_{L^\infty_tL_x^2\cap L_t^2H_x^1(\Omega\times [0,T])}
 +\|\d_\varepsilon\|_{L^\infty_tH^1_x(\Omega\times [0,T])}\Big\}\le C_1,\\
 &&\sup_{\varepsilon}\|\theta_\varepsilon\|_{L^\infty_t L^1_x\cap L^p_t W^{1,p}_x(\Omega\times [0,T])} 
 \le C_2(p), \ \forall\ p\in (1, \frac54),
 \end{eqnarray*}
 \begin{eqnarray}
 &&\int_{\Omega\times\{t\}} (|\u_\varepsilon|^2+|\nabla\d_\varepsilon|^2+\frac{2}{\varepsilon^2}{F(\d_\varepsilon)})
 +2\int_0^t\int_\Omega \big(\mu(\theta_\varepsilon)|\nabla\u_\varepsilon|^2
 +|\Delta\d_\varepsilon-\frac{1}{\varepsilon^2}\f(\d_\varepsilon)|^2\big)\nonumber\\
 &&\le \int_\Omega (|\u_0|^2+|\nabla\d_0|^2), \ \forall t\in [0, T], \label{global_bd1}
 \end{eqnarray}
  \begin{eqnarray}
 \int_{\Omega\times\{t\}} (|\u_\varepsilon|^2+|\nabla\d_\varepsilon|^2+\frac{2}{\varepsilon^2}{F(\d_\varepsilon)}
 +\theta_\varepsilon)
\le \int_\Omega (|\u_0|^2+|\nabla\d_0|^2+\theta_0), \ \forall t\in [0, T], \label{global_bd2}
 \end{eqnarray}
 and
 \begin{equation}\label{global_bd3}
 |\d_{\varepsilon}|\leq 1, \ {\d_{\varepsilon}^{3}\ge 0, \ \theta_\varepsilon\ge \essinf_{\Omega}\theta_0},
 \ \ {\rm{in}}\ \ \Omega\times [0,T].
 \end{equation}
 Applying the equation \eqref{eqn:1.3}, we can further deduce that
 \begin{equation}
 \sup_{\varepsilon}\Big\{\|\partial_t u_\varepsilon\|_{L^\frac43([0,T], H^{-1}(\Omega)}
 +\|\partial_t\d_\varepsilon\|_{L^\frac43([0,T], L^2(\Omega))}
 +\|\partial_t\theta_\varepsilon\|_{L^2([0,T], W^{-1,4}(\Omega)}\Big\}<C_3.
 \end{equation}
 Therefore, after passing to a subsequence, there exist
 $\u\in L^\infty_tL^2_x\cap L^2_tH^1_x(\Omega\times [0,T]),
 \d\in L^\infty_tH^1_x(\Omega\times [0,T]), \theta\in L^\infty_tL^1_x\cap L^p_tW^{1,p}_x(\Omega\times [0,T])$
 for $1<p<\frac54$ such that 
  \begin{equation}
   \left\{
     \begin{array}{ll}
       (\u_{\varepsilon}, \d_{\varepsilon})\rightarrow (\u, \d) & \text{ in }L^2(\Omega\times {  (0,T)}), \\
       (\nabla \u_{\varepsilon}, \nabla \d_{\varepsilon})\rightharpoonup(\nabla \u, \nabla \d) & \text{ in }{L^2}(\Omega\times {  (0,T)})
     \end{array}
     \right.
     \label{eqn:weakL2conv}
   \end{equation}
   as $\varepsilon\rightarrow 0$. Since
   $$\int_{\Omega\times [0,T]} F(\d)\le\lim_{\varepsilon}\int_{\Omega\times [0,T]} F(\d_\varepsilon)=0,$$
   we conclude that $|\d|=1$ a.e. in $\Omega\times [0,T]$. Sending $\varepsilon\to 0$ in the equations
   \eqref{eqn:1.3}$_{2,3}$, we obtain that 
   $$\nabla\cdot\u=0\ a.e.\ in \ \Omega\times [0,T],$$
   and
   $$(\partial_t\d+\u\cdot\nabla\d)\times \d=\nabla\cdot(\nabla \d\times \d)
   \ \ {\rm{weakly \ in}}\ \Omega\times [0,T],$$
 which, combined with the fact that $\d$ is $\mathbb S^2$-valued, implies
 that
 \begin{equation}\label{d-eqn}
 \partial_t\d+\u\cdot\nabla\d=\Delta\d+|\nabla\d|^2\d
  \ \ {\rm{weakly \ in}}\ \Omega\times [0,T].
 \end{equation}
 {Hence \eqref{weakd} holds.}
 
 To verify that $\u$ satisfies the equation \eqref{eqn:1.4}$_1$, we need to show
 that $\nabla \d_{\varepsilon}$ converges to $\nabla \d$ in $L^2_{\text{loc}}(\Omega\times(0,T))$.
 which makes sense of $\nabla\cdot(\nabla \d\odot \nabla \d)$. 
 We also need to justify the convergence of temperature equation 
 \eqref{eqn:1.4}$_4$.  For this purpose, we recall some basic notations and theorems in \cite{LW2} 
 that are needed in the proof.

 For any $0<a\leq 2$, $L_1$ and $L_2>0$, denote by
 $\mathcal{X}(L_1, L_2, a)$  the space that consists of weak solutions $\d_\varepsilon$ of
 \begin{equation*}
   \Delta \d_\varepsilon-\f_\varepsilon(\d_\varepsilon)=\tau_\varepsilon \text{ in }\Omega
   \label{}
 \end{equation*}
such that
 \begin{enumerate}
   \item {$|\d_\varepsilon|\leq 1$ and ${\color{red}\d_\varepsilon^{(3)}\geq -1+a}$ for $x$ a.e. in $\Omega$,}
   \item $E_\varepsilon(\d_\varepsilon)=\int_{\Omega}\frac{1}{2}|\nabla \d_\varepsilon|^2+3 F_\varepsilon(\d_\varepsilon) dx\leq L_1$,
   \item $\left\|\tau_\varepsilon\right\|_{L^2(\Omega)}\leq L_2$.
 \end{enumerate}
The following Theorem  concerning the $H^1$ pre-compactness of $\mathcal{X}(L_1, L_2, a)$ was 
shown by \cite{LW2}.
 \begin{Theorem}\label{LW2.0}
 For any $a\in (0,2]$, $L_1>0$ and $L_2>0$, the set $\mathcal{X}(L_1, L_2, a)$ is precompact in $H_{\loc}^1(\Omega; \R^3).$ Namely, if $\left\{ \d_\varepsilon \right\}$ is a sequence of maps in $\mathcal{X}(L_1,L_2, a)$, then there exists a map $\d\in H^1(\Omega; \mathbb{S}^2)$ such that, after passing to a possible subsequence, $\d_\varepsilon\rightarrow \d$ in $H_{\loc}^1(\Omega;\R^3).$
   \label{thm:X}
 \end{Theorem}
We also denote by $\mathcal{Y}(L_1, L_2, a)$ the space that consists of 
$\d\in H^1(\Omega, \mathbb{S}^2)$ that are so-called stationary approximated harmonic maps, 
more precisely, 
  \begin{equation}
  \begin{cases} \Delta \d+|\nabla\d|^2\d=\tau \ {\rm{in}}\ \Omega,\\
   \int_{\Omega}(\nabla\d\odot\nabla\d):\nabla\varphi-\frac{1}{2}|\nabla \d|^2\nabla\cdot \varphi+\left\langle \tau, \varphi\cdot \nabla \d\right\rangle=0,
   \label{eqn:epsuitable}
   \end{cases}
 \end{equation}
for any $\varphi\in C_0^\infty(\Omega;\R^3)$, and
\begin{enumerate}
  \item {$\d^{(3)}(x)\geq -1+a$ for $x$ a.e. in $\Omega$,}
  \item $E(\d)=\frac{1}{2}\int_{\Omega}|\nabla \d|^2dx\leq L_1,$
  \item $\left\|\tau\right\|_{L^2(\Omega)}\leq L_2.$
\end{enumerate}
The following $H^1$ pre-compactness of stationary approximated harmonic maps was also shown by \cite{LW2}.
\begin{Theorem}\label{LW2.1}
For any $a\in(0,2]$, $L_1>0$ and $L_2>0$, the set $\mathcal{Y}(L_1, L_2, a)$ is pre-compact in $H_{\rm loc}^1(\Omega;\mathbb{S}^2)$. Namely, if $\left\{ \d_i \right\}\subset \mathcal{Y}(L_1, L_2, a)$ is a sequence of  
stationary approximated harmonic maps, with tensor fields $\left\{ \tau_i \right\}$, then there exist $\tau\in L^2(\Omega, \R^3)$ and a stationary approximated harmonic map $\d\in \mathcal{Y}(L_1, L_2, a)$, with tensor field $\tau$, namely,
$$\Delta \d+|\nabla\d|^2\d=\tau \text{ in }\Omega,$$
such that after passing to a possible subsequence, $\d_i\rightarrow \d$ in $H_{\rm loc}^1(\Omega, \mathbb{S}^2)$ and $\tau_i\rightharpoonup \tau$ in $L^2(\Omega; \R^3)$. Moreover, $\d\in W_{\rm loc}^{2, 2}(\Omega,\mathbb{S}^2)$.
\label{thm:Y}
\end{Theorem}

Now we sketch the proof  the compactness of $\nabla \d_{\varepsilon}$ in $L^2_{loc}(\Omega\times [0,T])$. 
It follows from Fatou's lemma and \eqref{global_bd1}  that
 { \begin{equation*}
      \begin{array}{l}
	\int_{0}^{T}\liminf_{\varepsilon\rightarrow0}\int_{\Omega}|\Delta \d_\varepsilon-\f_\varepsilon(\d_\varepsilon)|^2\leq C_0.
      \end{array}
      \label{}
    \end{equation*}}
    We decompose $[0,T]$ into the sets of``good time slices'' and ``bad time slices''. {For $\Lambda\gg1$}, set
    $$\mathcal{G}_\Lambda^T:=\left\{ t\in[0,T]:\liminf_{\varepsilon\rightarrow0}\int_{\Omega}|\Delta \d_\varepsilon-f_\varepsilon(\d_\varepsilon)|^2(t)\leq \Lambda \right\},$$ 
    and
    $$\mathcal{B}_\Lambda^T:=[0,T]\setminus \mathcal{G}_\Lambda^T=\left\{ t\in [0,T]:\liminf_{\varepsilon\rightarrow0}{\int_{\Omega}}|\Delta \d_\varepsilon-\f_\varepsilon(\d_\varepsilon)|(t)>\Lambda \right\}.$$
 From Chebyshev's inequality, we have
 \begin{equation}
   |\mathcal{B}_\Lambda^T|\leq \frac{C_0}{\Lambda}.
   \label{eqn:Bad}
 \end{equation}
 For any $t\in \mathcal{G}_\Lambda^T$, set $\tau_\varepsilon(t)=\left( \Delta \d_\varepsilon-\f_\varepsilon(\d_\varepsilon) \right)(t).$ 
 Then  Lemma \ref{lemma:MAX1} and \ref{lemma:MAX2} imply
 that $\left\{ \d_\varepsilon(t) \right\}\subset \mathcal{X}(C_0, \Lambda, 1)$.
 Theorem \ref{thm:X} then implies that
\begin{equation*}
  \left\{
    \begin{array}{ll}
      \d_\varepsilon(t)\rightarrow \d(t) &\text{ in }H_{\loc}^1(\Omega),\\
      F_\varepsilon(\d_\varepsilon)\rightarrow 0 & \text{ in }L_{\loc}^1(\Omega),\\
      \tau_\varepsilon(t)\rightharpoonup \tau(t) & \text{ in }L^2(\Omega).
    \end{array}
    \right.
    \label{}
  \end{equation*}
 For any $\varphi\in C_0^\infty(\Omega;\R^3)$, {multiplying} $\tau_\varepsilon(t)$ by $\varphi\cdot \nabla\d_\varepsilon$ and integrating over $\Omega$ yields
 \begin{equation}
   \int_{\Omega}(\nabla\d_\varepsilon(t)\odot\nabla\d_\varepsilon(t)):\nabla\varphi
   -\big( \frac{1}{2}|\nabla \d_\varepsilon(t)|^2+F_\varepsilon(\d_\varepsilon(t)) \big)\nabla\cdot \varphi+\left\langle \tau_\varepsilon(t), \varphi\cdot \nabla \d_\varepsilon(t) \right\rangle=0.
   \label{eqn:epsuitable}
 \end{equation}
 Passing limit $\varepsilon\to 0$  in {\eqref{eqn:epsuitable}}, we get
 \begin{equation*}
   \int_{\Omega}(\nabla\d(t)\odot\nabla\d(t)):\nabla\varphi-\frac{1}{2}|\nabla \d(t)|^2\nabla\cdot \varphi+\left\langle \tau(t), \varphi\cdot \nabla \d(t) \right\rangle=0.
   \label{}
 \end{equation*}
 Hence $\d(t)\in\mathcal{Y}(C_0, \Lambda, 1)$ is a stationary approximated harmonic map.  Next we {want} to show that $\d_\varepsilon\rightarrow \d$ strongly in $L_t^2 H_x^1$. 
To see this, we claim that for any compact $K\subset\subset \Omega$,
 \begin{equation}
   \lim_{\varepsilon\rightarrow 0}\int_{K\times \mathcal{G}_\Lambda^T}|\nabla(\d_\varepsilon-\d)|^2=0.
   \label{eqn:contr1}
 \end{equation}
For, otherwise, there exist $\delta_0>0$, ${K}\subset\subset \Omega$ and $\varepsilon_i\rightarrow 0$ such that
 \begin{equation}
   {\int_{{K}\times \mathcal{G}_\Lambda^T}}|\nabla(\d_{\varepsilon_i}-\d)|^2\geq \delta_0.
    \label{eqn:prcontr1}
 \end{equation}
 From \eqref{eqn:weakL2conv}, we have
 \begin{equation}
   \lim_{\varepsilon_i\rightarrow0}\int_{{K}\times \mathcal{G}_\Lambda^T}|\d_{\varepsilon_i}-\d|^2=0.
   \label{eqn:prcontr2}
 \end{equation}
 By Fubini's theorem, \eqref{eqn:prcontr1} and \eqref{eqn:prcontr2}, 
 there would exist $t_i\in \mathcal{G}_\Lambda^T$ such that
 \begin{equation*}
   \left\{
     \begin{array}{l}
       \lim_{\varepsilon_i\rightarrow0}\int_{{K}}|\d_{\varepsilon_i}({t_i})-\d(t_i)|^2=0, \\
       \int_{{K}}|\nabla(\d_{\varepsilon_i}(t_i)-\d(t_i))|^2\geq \frac{2\delta_0}{T}.
     \end{array}
     \right.
     \label{}
   \end{equation*}
Thus $\left\{ \d_{\varepsilon_i}(t_i) \right\}\subset \mathcal{X}(C_0, \Lambda, 1)$ 
and $\left\{ \d(t_i) \right\}\subset \mathcal{Y}(C_0, \Lambda, 1)$. It follows from Theorem \ref{thm:X} and 
Theorem \ref{thm:Y} that there exist $\d_1, \d_2\in \mathcal{Y}(C_0, \Lambda, 1)$ such that
   \begin{equation*}
     \d_{\varepsilon_i}(t_i)\rightarrow \d_1 \text{ and }\d(t_i)\rightarrow\d_2 \text{ strongly in } H^1(\Omega).
     \label{}
   \end{equation*}
   Therefore we would have 
   \begin{equation*}
     \int_{{K}}|\nabla(\d_1-\d_2)|^2=\lim_{i\rightarrow\infty}\int_{{K}}|\nabla\left( \d_\varepsilon(t_i)-\d(t_i) \right)|^2\geq \frac{2\delta_0}{T},
     \label{}
   \end{equation*}
and
  $$\int_{{K}}|\d_1-\d_2|^2=\lim_{i\rightarrow \infty}\int_{{K}}|\d_{\varepsilon_i}(t_i)-\d(t_i)|^2=0.$$
This is clearly impossible. Thus the claim is true. 

We can also follow the proof of Theorem \ref{thm:X} in \cite{LW2} to conclude
that the small energy regularity criteria holds for every $(x,t)\in K\times\mathcal{G}_\Lambda^T$ 
so that a finite covering argument, together with estimates for Claim 4.5 in \cite{LW2},
yields
  \begin{equation}
      \lim_{\varepsilon\rightarrow 0}\int_{K\times \mathcal{G}_\Lambda^T}F_\varepsilon(\d_\varepsilon)=0.
      \label{eqn:Fepsilon}
  \end{equation}
Hence we have that
$$\lim_{\varepsilon\rightarrow 0}\Big[\|\d_\varepsilon-\d\|_{L^2_t H_x^1(K\times \mathcal{G}_\Lambda^T)}^2+\int_{K\times \mathcal{G}_\Lambda^T}F_\varepsilon(\d_\varepsilon)\Big]=0.$$
On the other hand, it follows from \eqref{global_bd1} and \eqref{eqn:Bad} that
\begin{eqnarray*}
&&\left\|\d_\varepsilon-\d\right\|_{L_t^2H_x^1(\Omega\times \mathcal{B}_\Lambda^T)}^2
+\int_{\Omega\times \mathcal{B}_\Lambda^T}F_\varepsilon(\d_\varepsilon)\\
&&\leq
  C \Big(\sup_{t>0}\int_{\Omega} (|\u_\varepsilon|^2+|\nabla\d_\varepsilon|^2+F_\varepsilon(\d_\varepsilon))\Big)
  \big|\mathcal{B}_\Lambda^T\big|\leq \frac{C}{\Lambda}.
  \label{}
\end{eqnarray*}
Therefore, we would arrive at
\begin{equation*}
  \lim_{\varepsilon\rightarrow 0}\Big[ \left\|\d_\varepsilon-\d\right\|_{L_t^2H_x^1(K\times[0,T])}^2+\int_{K\times [0,T]}F_\varepsilon(\d_\varepsilon) \Big]\leq \frac{C}{\Lambda}.
  \label{}
\end{equation*}
Sending $\Lambda\to\infty$ yields that 
\begin{equation*}
   \lim_{\varepsilon\rightarrow 0}\Big[ \left\|\d_\varepsilon-\d\right\|_{L_t^2H_x^1(K\times[0,T])}^2+\int_{K\times [0,T]}F_\varepsilon(\d_\varepsilon) \Big]=0.
  \label{}
\end{equation*}
Therefore we can conclude that $\u$ solves the equation \eqref{eqn:weaku}, provided
we can verify that $\mu(\theta_\varepsilon)\nabla \u_\varepsilon
\rightharpoonup \mu(\theta)\nabla\u$ weakly in $L^2(\Omega\times [0,T])$, which will be verified below.

Next we turn to the convergence of $\theta_{\varepsilon}$.
For $\alpha\in (0, 1)$, set $H(\theta_{\varepsilon})=(1+\theta_{\varepsilon})^\alpha$.
Then from (\ref{eqn:approximateEntropy}) we have
   \begin{eqnarray}
 && \partial_t(1+\theta_{\varepsilon})^\alpha+\u_{\varepsilon}\cdot \nabla(1+\theta_{\varepsilon})^\alpha
 \nonumber\\
 &&\geq-\div\left( \alpha(1+\theta_{\varepsilon})^{\alpha-1}\q_{\varepsilon} \right)
  +\alpha(1+\theta_{\varepsilon})^{\alpha-1}\left( \mu(\theta_{\varepsilon})|\nabla \u_{\varepsilon}|^2+|\Delta \d_{\varepsilon}-\f_\varepsilon(\d_{\varepsilon})|^2 \right)\nonumber\\
&&  +\alpha(\alpha-1)(1+\theta_{\varepsilon})^{\alpha-2}\q_{\varepsilon}\cdot \nabla \theta_{\varepsilon}.
  \label{eqn:approximateEnt}
\end{eqnarray}
Integrating \eqref{eqn:approximateEnt} over ${\Omega\times[0,T]}$, 
by the assumption \eqref{eqn:coeffbound} on $\mu$, and the bound
\eqref{global_bd1} on $\u_{\varepsilon}, \d_{\varepsilon}$ and $\theta_{\varepsilon}$, 
we can derive that
   \begin{equation*}
     \sup_{\varepsilon>0} \sup_{0<t<{T}}\int_{\Omega}(1+\theta_{\varepsilon})^{\alpha-2}|\nabla \theta_{\varepsilon}|^2  <\infty.
     \label{}
   \end{equation*}
Therefore we conclude that $\theta_{\varepsilon}^{\frac{\alpha}{2}}\in L_t^2 H_x^1$
and $\theta_{\varepsilon}\in L^\infty_t L_x^1$ are uniformly bounded.  
By interpolation, we would have that for {$1\leq p<5/4$},
   \begin{equation*}
     \sup_{\varepsilon>0}\left\|\theta_{\varepsilon}\right\|_{L_t^p W_x^{1, p}(\Omega\times [0,T])}<\infty.
     \label{}
   \end{equation*}
   From the equation (\ref{eqn:modified-system})$_4$, we have that for {$1\leq q<\frac{30}{23}$},
   \begin{align*}
     \sup_{\varepsilon>0} \left\|\partial_t \theta_{\varepsilon}\right\|_{L^1_tW^{-1,q}_x}&\leq \sup_{\varepsilon>0}\Big(C\|\u_{\varepsilon}\theta_{\varepsilon}\|_{L^q_tL^q_x}+C\|\nabla \theta_{\varepsilon}\|_{L^q_tL_x^q}\\
     &+C\left\||\nabla \u_{\varepsilon}|^2+|\Delta \d_{\varepsilon}-\f_\varepsilon(\d_{\varepsilon})|^2\right\|_{L^1_tL^1_x}\Big)\\
     &\leq C\sup_{\varepsilon>0}\Big(\left\|\u_{\varepsilon}\right\|_{L^{{\frac{10}{3}}}_tL^{{\frac{10}{3}}}_x}\left\|\theta_{\varepsilon}\right\|_{L^{\frac{10q}{10-3q}}_tL_x^{\frac{10q}{10-3q}}}+\left\|\nabla \theta_{\varepsilon}\right\|_{L_t^qL_x^q}\Big)+C \\
     &<\infty.
     \label{}
   \end{align*}
Hence, by Aubin-Lions' compactness Lemma \cite{Simon1996} again, 
up to a subsequence, there exists $\theta\in L_t^\infty L_x^1\cap L_t^p W_x^{1, p}$ for $1\le p<\frac54$
such that
   \begin{equation*}
     \left\{
       \begin{array}{ll}
	 \theta_{\varepsilon}\rightarrow \theta & \text{ in }L^p(\Omega\times {(0,T)}),\\
	 \nabla \theta_{\varepsilon}\rightharpoonup \nabla \theta & \text{ in }L^p(\Omega\times {(0,T)}),
       \end{array}
       \right.
       \label{}
     \end{equation*}
as ${\varepsilon\rightarrow 0}$.

After taking another subsequence, we may assume that $(\u_{\varepsilon}, \d_{\varepsilon}, \theta_{\varepsilon})$ converge to $(\u, \d, \theta)$ a.e. in $\Omega\times [0, T].$

Since $\{\mu(\theta_\varepsilon)\}$ is uniformly bounded in $L^\infty(\Omega\times[0,T])$,
$\mu(\theta_{\varepsilon})\rightarrow \mu(\theta)$ a.e. in {$\Omega\times[0,T]$}
 and $\nabla \u_{\varepsilon}\rightharpoonup \nabla \u$ in $L^2({\Omega\times[0,T]})$, it follows
 that 
 $$\mu(\theta_{\varepsilon})\nabla \u_{\varepsilon}\rightharpoonup \mu(\theta)\nabla \u
 \ {\rm{ in }}\ L^2({\Omega\times[0,T]}).$$
 Thus we verify that \eqref{eqn:weaku} holds.

Taking the $L^2$ inner product of 
$\u_\varepsilon$, $\d_\varepsilon$, $\theta_\varepsilon$ 
in (\ref{eqn:modified-system}) with respect to $\u_{\varepsilon}, {-\Delta \d_{\varepsilon}+\f_\varepsilon(\d_{\varepsilon})}, 1$, and adding the resulting equations together,
we have the following energy law:
 \begin{equation}
   \frac{d}{dt}\int_{\Omega}\Big( \frac{1}{2}|\u_{\varepsilon}|^2+\frac{1}{2}|\nabla \d_{\varepsilon}|^2+F_\varepsilon(\d_{\varepsilon})+\theta_{\varepsilon} \Big)=0.
   \label{eqn:RegTotalEnergyConser}
 \end{equation}
Taking $\varepsilon\to 0$, this implies that $|\d|=1$ and 
\begin{equation*}
\int_\Omega \big(\frac12|\u|^2+\frac12|\nabla\d|^2+\theta\big)(t)
\leq \int_\Omega \big(\frac12{|\u_0|^2}+\frac12{|\nabla \d_0|}+\theta_0\big), \ \forall 0\le t\le T.
  \label{}
\end{equation*}
Hence the global energy inequality \eqref{weaktotal} holds.

It remains to show that \eqref{eqn:Entropy} follows
by passing limit $\varepsilon\to 0$ in \eqref{eqn:RelaxedEntropy}. 
This can be done exactly as in the last part of the previous section. 
For any smooth, nondecreasing, concave function $H$, and $\psi\in C_0^\infty(\overline{\Omega}\times[0,T))$,
recall from \eqref{entropy1} that
\begin{equation}
  \begin{split}
    &\int_{0}^{T}\int_{\Omega}\left( H(\theta_\varepsilon)\partial_t \psi+(H(\theta_\varepsilon)\u_\varepsilon-H'(\theta_\varepsilon)\q_\varepsilon )\cdot \nabla\psi\right)\\
    &\leq -\int_{0}^{T}\int_{\Omega}[H'(\theta_\varepsilon)(\mu(\theta_\varepsilon)|\nabla u_\varepsilon|^2+|\Delta\d_\varepsilon -\f_\varepsilon(\d_\varepsilon)|^2)-H''(\theta_\varepsilon)\q_\varepsilon\cdot \nabla\theta_\varepsilon]\psi\\
    & \quad -\int_{\Omega}H(\theta_0)\psi(\cdot , 0).
  \end{split}
  \label{entropy616}
\end{equation}
Assume $H(0)=0$. Then the concavity of $H$, $0\leq H'(\theta_{\varepsilon})\leq H'(\essinf_{\Omega}\theta_0)$,
 and the uniform bound on $\theta_\varepsilon$ imply that 
  $$\{H(\theta_\varepsilon)\} \ \mbox{is bounded in} \ \ L^\infty_t L^1_x
\cap L^p_t W^{1,p}_x(\Omega\times [0,T]), \ \forall 1<p<\frac54.
$$
Together with the bounds on $\u_\varepsilon, \d_\varepsilon$, and \eqref{entropy616}, we have that
\begin{eqnarray*}
	&&\int_0^T\int_\Omega H''(\theta_\varepsilon) \q_\varepsilon\cdot\nabla\theta_\varepsilon\psi\\
	&&=\int_0^T\int_\Omega (|\sqrt{-H''(\theta_\varepsilon) k(\theta_\varepsilon)\psi}\nabla \theta_\varepsilon|^2
	+|\sqrt{-H''(\theta_\varepsilon) h(\theta_m)\psi}(\nabla \theta_\varepsilon\cdot\d_\varepsilon)|^2)
\end{eqnarray*}
is uniformly bounded. By an argument similar to \eqref{conv1}, we can show that
\begin{equation}
  \int_{0}^{T}\int_{\Omega}-H''(\theta)\q\cdot\nabla \theta \psi\leq \liminf_{\varepsilon\to0} 
  \int_{0}^{T}\int_{\Omega}-H''(\theta_\varepsilon)\q_\varepsilon\cdot \nabla\theta_\varepsilon \psi.
  \label{conv617}
\end{equation}
Observe that 
\begin{equation*}
\begin{split}
  \Delta \d_\varepsilon-\f_\varepsilon(\d_\varepsilon)=\partial_t \d_\varepsilon+\u_\varepsilon\cdot \nabla\d_\varepsilon
  \rightharpoonup \partial_t \d+\u\cdot \nabla\d=\Delta\d+|\Delta \d|^2 \d
  \quad\text{in}\quad L^2(\Omega\times[0,T]),
  \end{split}
\end{equation*}
and $\left\{ H'(\theta_\varepsilon) \right\}$ is uniformly bounded in $L^\infty(\Omega\times[0,T])$.
It follows from the lower semicontinuity that
\begin{eqnarray}\label{conv618}
&&\int_0^T\int_\Omega\big[H'(\theta)(\mu(\theta)|\nabla\u|^2+|\Delta\d+|\nabla\d|^2 \d|^2)\psi\nonumber\\
&&\le\liminf_{\varepsilon\to 0}
\int_0^T\int_\Omega\big[H'(\theta_\varepsilon)(\mu(\theta_\varepsilon)|\nabla\u_\varepsilon|^2+|\Delta\d_\varepsilon-\f_\varepsilon(\d_\varepsilon)|^2)\psi.
\end{eqnarray}
On the other hand, since 
$$H(\theta_\varepsilon)\to H(\theta), \  H(\theta_\varepsilon)\u_\varepsilon\to H(\theta)\u \ \ {\rm{in}}\ \ L^1(\Omega\times [0,T]),$$
and
$$H'(\theta_\varepsilon)\q_\varepsilon\rightharpoonup H'(\theta) \q \ \ {\rm{in}}\ \ L^1(\Omega\times [0,T]),$$
we have
\begin{eqnarray}\label{conv619}
&&\int_0^T\int_\Omega \big(H(\theta)\partial_t\psi+(H(\theta)\u-H'(\theta)\q)\cdot\nabla\psi\big)\nonumber\\
&&=\lim_{\varepsilon\to 0}
\int_0^T\int_\Omega \big(H(\theta_\varepsilon)\partial_t\psi+(H(\theta_\varepsilon)\u_\varepsilon-H'(\theta_\varepsilon)\q_\varepsilon)\cdot\nabla\psi\big).
\end{eqnarray}
Therefore \eqref{weakEnt} follows by passing $\varepsilon\to0$ in \eqref{entropy616} and applying
\eqref{conv617}, \eqref{conv618}, and \eqref{conv619}.  This completes the construction of a global weak solution
to \eqref{eqn:1.4}. 
\end{proof}

\bigskip
\noindent{\bf Acknowledgments}. The paper was complete while the second author 
was a visiting PhD student of Purdue University. She would like to express her gratitude to
the Department of Mathematics for the hospitality. Both the first author and third author are partially
supported by NSF grant 1764417.

\bigskip

      \bibliographystyle{amsplain}
       %\bibliography{WeakNonisothermal.bib}
       %\notice{}

\end{document}